\documentclass{amsart}
\usepackage{amssymb,wasysym,url}
\usepackage[all]{xy}

\newcommand{\gr}{{\operatorname{gr}\nolimits}}

\newcommand{\kar}{{\operatorname{char}\,}}

\renewcommand{\Im}{\operatorname{Im}\nolimits}

\newcommand{\rrad}{\mathfrak{r}}

\newcommand{\gldim}{\operatorname{gldim}\nolimits}

\newcommand{\Ext}{\operatorname{Ext}\nolimits}

\newcommand{\op}{{\operatorname{op}\nolimits}}

\renewcommand{\L}{\Lambda}
\newcommand{\Z}{{\mathbb Z}}
\newcommand{\A}{{\mathcal A}}

\newcommand{\HH}{\operatorname{HH}\nolimits}

\newtheorem{lem}{Lemma}[section]
\newtheorem{prop}[lem]{Proposition}

\newtheorem{thm}[lem]{Theorem}
\newtheorem*{lemma}{Lemma}
\newtheorem*{theorem}{Theorem}
\theoremstyle{definition}

\newtheorem*{remark}{Remark}

\begin{document}

\title[Radical cube zero weakly symmetric algebras]
{Radical cube zero weakly symmetric algebras and support varieties} 
\author[Erdmann]{Karin Erdmann}
\address{Karin Erdmann\\
Mathematical Institute\\ 
24--29 St.\ Giles\\
Oxford OX1 3LB\\ 
England}
\email{erdmann@maths.ox.ac.uk}
\author[Solberg]{\O yvind Solberg}
\address{\O yvind Solberg\\
Institutt for matematiske fag\\
NTNU\\ 
N--7491 Trondheim\\ 
Norway}
\email{Oyvind.Solberg@math.ntnu.no}
\thanks{The authors acknowledge support from EPSRC grant
  EP/D077656/1 and NFR Storforsk grant no.\ 167130.} 

\begin{abstract}
One of our main results is a classification all the weakly symmetric
radical cube zero finite dimensional algebras over an algebraically
closed field having a theory of support via the Hochschild cohomology
ring satisfying Dade's Lemma. Along the way we give a characterization
of when a finite dimensional Koszul algebra has such a theory of
support in terms of the graded centre of the Koszul dual.
\end{abstract}

\date{\today}
\maketitle

\section*{Introduction}

The support variety of a module, if it exists, is a powerful
invariant. For group algebras of finite groups, or cocommutative Hopf
algebras it is defined in terms of the maximal ideal spectrum of the
group cohomology ring. For a more general finite-dimensional algebra
$\Lambda$, instead of group cohomology one can take a subalgebra of
the Hochschild cohomology, provided it has suitable finite generation
properties (see Section \ref{section:1} for details). It was shown in
\cite{EHSST}, that if these hold and $\Lambda$ is self-injective, then
many of the standard results from the theory of support varieties for
finite groups generalize to this situation.

The question is now to understand whether or not these finite
generation properties hold for given classes of algebras.  Since
Hochschild cohomology is difficult to calculate explicitly, one would
rather not do this and have other ways to detect finite generation.

In this paper we present a method, for Koszul algebras, which gives a
criterion in terms of the Koszul dual, to show that the finite
generation condition holds. We denote this condition by
\textbf{(Fg)}. We apply this method to a general weakly symmetric
algebra with radical cube zero, and also to quantum exterior algebras.

In \cite{B}, D.\ Benson characterized the rate of growth of
resolutions for weakly symmetric algebras with radical cube zero.  We
use his results, and we show that almost all such algebras which are
of tame representation type satisfy the finite generation
hypothesis. There are only two exceptional cases, where a deformation
parameter $q$ appears and whenever $q$ is not a root of unity. It is
clear from \cite{B} that the algebras in Benson's list which have wild
type, cannot satisfy the finite generation condition. Using \cite{D}
it follows easily that all weakly symmetric (in fact all
selfinjective) algebras of finite representation type over an
algebraically closed field satisfy the finite generation
condition. The precise answer is the following (see Theorem
\ref{solberg-thm:benson} for the definition of the quivers involved).
\begin{theorem}
Let $\Lambda$ be a finite dimensional symmetric algebra over an
algebraically closed field with radical cube zero and radical square
non-zero. Then $\Lambda$ satisfies \emph{\textbf{(Fg)}} if and only if
$\Lambda$ is of finite representation type, $\Lambda$ is of type
$\widetilde{\mathbb{D}}_n$ for $n\geq 4$, $\widetilde{Z}_n$ for $n>0$,
$\widetilde{D\mathbb{Z}}_n$, $\widetilde{\mathbb{E}}_6$,
$\widetilde{\mathbb{E}}_7$, $\widetilde{\mathbb{E}}_8$, or $\Lambda$
is of type $\widetilde{\Z}_0$ or $\widetilde{\mathbb{A}}_n$ when $q$ is
a root of unity.
\end{theorem}
In addition we show that a quantum exterior algebra satisfies
\textbf{(Fg)} if and only if all deformation parameters are roots of
unity.  This is generalized in \cite{BO}. Furthermore, this result,
together with \cite{EHSST} also generalizes almost all of Theorem 2.2
in \cite{E}.

As observed in \cite{BGMS}, even ``nice'' selfinjective algebras do
not necessarily satisfy \textbf{(Fg)}, in spite of sharing many of the
same representation structural properties with algebras having
\textbf{(Fg)}.  Hence, whether or not a finite dimensional algebra has
the property \textbf{(Fg)} is a more subtle question than one first
would expect. However, the list of classes of algebras satisfying
\textbf{(Fg)} is quite extensive: (i) any block of a group ring of a
finite group \cite{Ev,Go,V}, (ii) any block of a finite dimensional
cocommutative Hopf algebra \cite{FrS}, (iii) in the commutative
setting for a complete intersection \cite{Gu}, (iv) any exterior
algebra, (v) all finite dimensional selfinjective algebras over an
algebraically closed field of finite representation type \cite{D},
(vi) quantum complete intersections \cite{BO}, (vii) all Gorenstein
Nakayama algebras (announced by H.\ Nagase), (viii) any finite
dimensional pointed Hopf algebra, having abelian group of group-like
elements, under some mild restrictions on the group order \cite{MPSW}.

A weakly symmetric algebra is, in particular, a finite dimensional
selfinjective algebra. In \cite{Sm} it is shown that a finite
dimensional Koszul algebra $\L$ over a field $k$ with degree zero part
isomorphic to $k$, is selfinjective with finite complexity if and only
if the Koszul dual $E(\L)$ is an Artin-Schelter regular Koszul
algebra. This was extended in \cite{MV2} to finite dimensional Koszul
algebras $\L$ over a field $k$ with $\L_0\simeq k^n$ for some positive
integer $n$. Then it is natural to say that a (non-connected) Koszul
$k$-algebra $R=\oplus_{i\geq 0}R_i$ is an \emph{Artin-Schelter regular
  algebra of dimension $d$}, if
\begin{enumerate}
\item[(i)] $\dim_k R_i<\infty$ for all $i\geq 0$, 
\item[(ii)] $R_0\simeq k^n$ for some positive integer $n$, 
\item[(iii)] $\gldim R=d$,
\item[(iv)] the Gelfand-Kirillov dimension of $R$ is finite, 
\item[(v)] for all simple graded $R$-modules $S$ we have
  \[\Ext^i_R(S,R)\simeq \begin{cases} (0), & i\neq d,\notag\\ S', & i=d
  \text{\ and some simple graded $R^\op$-module $S'$}.\notag
\end{cases}\]
\end{enumerate}
As in \cite{M} it follows that classifying all selfinjective Koszul
algebras of finite complexity $d$ and Loewy length $m+1$ (up to
isomorphism) is the same as classifying Artin-Schelter regular Koszul
algebras with Gelfand-Kirillov dimension $d$ and global dimension $m$
(up to isomorphism). Moreover, if $R$ is such an Artin-Schelter
regular Koszul algebra, then $\Ext^d_{R^\op}(\Ext^d_R(S,R),R)\simeq S$
and consequently $\Ext^d_R(S,-)$ is a ``permutation'' of the simple
graded $R$- and $R^\op$-modules. The selfinjective algebra $\L$ being
weakly symmetric, corresponds to the property that $R=E(\L)$, the
Koszul dual of $\L$, satisfies $\Ext^d_R(S,R)\simeq S^\op$ for all
simple graded $R$-modules $S$. Hence, the weakly symmetric algebras
with radical cube zero found in \cite{B} of finite complexity classify
all the Artin-Schelter regular Koszul algebras of dimension $2$, where
$\Ext^2_R(-,R)$ is the ``identity permutation''.  Within this class of
algebras, we classify those that are finitely generated modules over
their centres, such that $\Ext^2_R(-,R)$ is the identity permutation.
A similar classification is also carried out for Artin-Schelter
regular Koszul algebras of dimension $1$ (see Proposition
\ref{prop:radical-square-zero}). These algebras correspond to radical
square zero algebras, so after having stated this result, we exclude
this class of algebras from the discussion of weakly symmetric radical
cube zero algebras.

An early version of this paper was called \emph{Finite generation of
the Hochschild cohomology ring of some Koszul algebras}. All the
results that existed in that early version are included in full in the
current paper. 

The authors acknowledge the use of the Gr\"obner basis program GRB by
E.\ L.\ Green \cite{Gr} in the experimental stages of this paper. 

\section{Background and preliminary results}\label{section:1}

Throughout let $\L$ be a finite dimensional algebra over an
algebraically closed field $k$ with Jacobson radical
$\mathfrak{r}$. Cohomological support varieties for finite dimensional
modules over $\L$ using the Hochschild cohomology ring $\HH^*(\L)$ of
$\L$ were introduced in \cite{SS} and further studied in
\cite{EHSST}. It follows that $\L$ has a good theory of cohomological
support varieties via $\HH^*(\L)$ of $\L$, if $\HH^*(\L)$ is
Noetherian and $\Ext^*_\L(\L/\rrad,\L/\rrad)$ is a finitely generated
$\HH^*(\L)$-module.  Denote this condition by \textbf{(Fg)}. The aim
of this paper is to characterize when a weakly symmetric algebra $\L$
with $\rrad^3=(0)$ satisfies \textbf{(Fg)}.

For the algebras $\L$ we consider, it is well-known that
$\L\simeq kQ/I$ for some finite quiver $Q$ and some ideal $I$ in
$kQ$, up to Morita equivalence. Furthermore there is a homomorphism of
graded rings
\[\varphi_M\colon \HH^*(\L)\to
\Ext^*_\L(M,M)=\oplus_{i\geq
  0}\Ext^i_\L(M,M)\] 
\sloppy for all $\L$-modules $M$, with $\Im\varphi_M
\subseteq Z_{\gr}(\Ext^*_\L(M,M))$
(see \cite{SS,Y}).  Here
$Z_{\gr}(\Ext^*_\L(M,M))$ denotes the
graded centre of $\Ext^*_\L(M,M)$. The graded centre of a
graded ring $R$ is generated by
\[\{ z\in R\mid rz=(-1)^{|r||z|} zr, \text{\ $z$
  homogeneous,\ }\forall n\geq 0, \forall r\in R_n\},\] 
where $|x|$ denotes the degree of a homogeneous element $x$ in $R$.

Weakly symmetric algebras are selfinjective algebras where all
indecomposable projective modules $P$ have the property that
$P/\mathfrak{r}P\simeq \operatorname{Soc}(P)$.  All selfinjective
algebras $\L$ of finite representation type are shown to be periodic
algebras \cite{D}, meaning that
$\Omega^n_{\L\otimes_k\L^{\operatorname{op}}}(\L)\simeq \L$ for some
$n\geq 1$. It is easy to see that all periodic algebras $\L$ satisfy
\textbf{(Fg)}. Furthermore, for selfinjective algebras $\L$ with
$\mathfrak{r}^3=(0)$ and $\rrad^2\neq (0)$ we have the following result.
\begin{thm}[\cite{G,MV}]
Let $\L$ be a selfinjective algebra with $\rrad^3=(0)$ and $\rrad^2\neq
(0)$. Then $\L$ is Koszul if and only if $\L$ is of infinite
representation type.
\end{thm}
Hence in our study of weakly symmetric algebras $\L$ with
$\mathfrak{r}^3=(0)$, we can concentrate on infinite representation
type and consequently Koszul algebras. For Koszul algebras $\L$
the homomorphism of graded rings from $\HH^*(\L)$
to the Ext-algebra of the simple modules has an even nicer property
than for general algebras, as we discuss next.

For a quotient of a path algebra $\L=kQ/I$, given by a quiver $Q$ with
relations $I$ over a field $k$, it was independently observed by
Buchweitz and Green-Snashall-Solberg, that when $\L$ is a Koszul
algebra, the image of the natural map from the Hochschild cohomology
ring $\HH^*(\L)$ to the Koszul dual $E(\L)$ is equal to the graded
centre $Z_\gr(E(\L))$ of $E(\L)$. Here $E(\L)=\oplus_{i\geq
  0}\Ext^i_\L(\L_0,\L_0)$, where $\L_0$ is the degree $0$ part of
$\L$. This isomorphism was obtained by Buchweitz as a part of a more
general isomorphism between the Hochschild cohomology ring of $\L$ and
the graded Hochschild cohomology ring of the Koszul dual.  This
isomorphism has since been generalized by Keller \cite{K}.  The
statement from \cite{BGSS} reads as follows.

\begin{thm}[\cite{BGSS}]\label{thm:bgss}
Let $\L=kQ/I$ be a Koszul algebra. Then the image of the natural map
$\varphi_{\L_0}\colon \HH^*(\L)\to E(\L)$ is the graded centre
$Z_\gr(E(\L))$.
\end{thm}
This enables us to characterize when \textbf{(Fg)} holds for a finite
dimensional Koszul algebra $\L$ over an algebraically closed
field. Let $\L=kQ/I$ be a path algebra over an algebraically closed
field $k$. In \cite{EHSST} the following conditions are crucial with
respect to having a good theory of cohomological support varieties
over $\L$:
\begin{enumerate}
\item[\textbf{Fg1}:] there is a commutative Noetherian graded
subalgebra $H$ of $\HH^*(\L)$ with $H^0=\HH^0(\L)$,
\item[\textbf{Fg2}:] $\Ext^*_\L(\L/\rrad,\L/\rrad)$ is a finitely generated
  $H$-module, where $\rrad$ denotes the Jacobson radical of $\L$.
\end{enumerate}
In \cite[Proposition 5.7]{S} it is shown that these conditions are
satisfied if and only if the condition \textbf{(Fg)} holds for $\L$.
Suppose that \textbf{(Fg)} is satisfied for $\L$.  Since
$Z_\gr(E(\L))$ is an $\HH^*(\L)$-submodule of $E(\L)$, the Hochschild
cohomology ring $\HH^*(\L)$ is Noetherian and $E(\L)$ a finitely
generated $\HH^*(\L)$-module, we infer that $Z_\gr(E(\L))$ is a
finitely generated $\HH^*(\L)$-module as well. Hence $Z_\gr(E(\L))$ is
a Noetherian algebra and $E(\L)$ is clearly finitely generated as a
$Z_\gr(E(\L))$-module.  If $\L$ is a finite dimensional Koszul
algebra, the converse is also true. Suppose that $\L$ is a finite
dimensional Koszul algebra with $Z_\gr(E(\L))$ a Noetherian algebra
and $E(\L)$ a finitely generated module over $Z_\gr(E(\L))$. Then
$Z_\gr(E(\L))$ contains a commutative Noetherian even-degree graded
subalgebra $\widetilde{H}$, such that $Z_\gr(E(\L))$ is a finitely
generated module over $\widetilde{H}$. Let $H$ be an inverse image of
$\widetilde{H}$ in $\HH^*(\L)$. Then $H$ is (can be chosen to be) a 
commutative Noetherian graded subalgebra of $\HH^*(\L)$. Therefore the
conditions \textbf{(Fg1)} and \textbf{(Fg2)} are satisfied, and
consequently \textbf{(Fg)} holds true for $\L$ by \cite[Proposition
  5.7]{S}. Hence we have the following.

\begin{thm}\label{thm:Fgcharacterization}
Let $\L=kQ/I$ be a finite dimensional algebra over an algebraically
closed field $k$, and let $E(\L)=\Ext^*_\L(\L/\mathfrak{r},
\L/\mathfrak{r})$.
\begin{enumerate}
\item[(a)] If $\L$ satisfies \emph{\textbf{(Fg)}}, then
  $Z_{\gr}(E(\L))$ is Noetherian and $E(\L)$
  is a finitely generated $Z_{\gr}(E(\L))$-module. 
\item[(b)] When $\L$ is Koszul, then the converse implication
  also holds, that is, if $Z_{\gr}(E(\L))$ is
  Noetherian and $E(\L)$ is a finitely generated
  $Z_{\gr}(E(\L))$-module, then $\L$ satisfies
  \emph{\textbf{(Fg)}}. 
\end{enumerate}
\end{thm}
In part (b) it is enough to find a commutative Noetherian graded
subring of $Z_\gr(E(\L))$ over which $E(\L)$ is a finitely generated
module. 

Our main focus in this paper is radical cube zero algebras. However,
let us illustrate the above result on radical square zero algebras.
Let $\L=kQ/J^2$ be a finite dimensional radical square zero algebra
over a field $k$. Hence $J$ is the ideal generated by the arrows in
$Q$. Since all quadratic monomial algebras are Koszul \cite{GZ}, $\L$
is a Koszul algebra and $E(\L)=kQ^\op$. If $Q$ does not have any
oriented cycles, the global dimension of $\L$ is finite and there is
no interesting theory of support varieties via the Hochschild
cohomology ring. So, assume that $Q$ has at least one oriented
cycle. Consequently $kQ^\op$ is an infinite dimensional
$k$-algebra. In addition we have that
\[Z_\gr(E(\L))=
\begin{cases}
k, & \text{if $Q$ is not
an oriented cycle ($\widetilde{\mathbb{A}}_n$),}\\
k[T], & \text{otherwise},
\end{cases}\]
see for example \cite{CB}. 

The following result is an immediate consequence of the above. 
\begin{prop}\label{prop:radical-square-zero}
Let $\L=kQ/J^2$ for some quiver $Q$ with at least one oriented cycle
and a field $k$. Then $\L$ satisfies \emph{\textbf{(Fg)}} if and only
if $\L$ is a radical square zero Nakayama algebra.
\end{prop}

Let $\L$ be a weakly symmetric algebra with $\mathfrak{r}^3=(0)$ and
$\rrad^2\neq (0)$. Denote by $\{S_1,\ldots,S_n\}$ all the
non-isomorphic simple $\L$-modules, and let $E_\L$ be the $n\times
n$-matrix given by $(\dim_k\Ext^1_\L(S_i,S_j))_{i,j}$. These algebras
are classified in \cite{B}, and among other things the following is
proved there.
\begin{thm}[\cite{B}]\label{solberg-thm:benson}
Let $\L$ be a finite dimensional indecomposable basic weakly symmetric
algebra over an algebraically closed field $k$ with
$\mathfrak{r}^3=(0)$ and $\rrad^2\neq (0)$. Then the matrix $E_\L$ is
a symmetric matrix, and the eigenvalue $\lambda$ of $E_\L$ with
largest absolute value is positive.
\begin{enumerate}
\item[(a)] If $\lambda > 2$, then the dimensions of the modules in a
  minimal projective resolution of any finitely generated
  $\L$-module has exponential growth.
\item[(b)] If $\lambda = 2$, then the dimensions of the modules in a
  minimal projective resolution of any finitely generated $\L$-module
  are either bounded or grow linearly. The matrix $E_\L$ is the
  adjacency matrix of a Euclidean diagram $\widetilde{\mathbb{A}}_n$
  for $n\geq 1$, $\widetilde{\mathbb{D}}_n$ for $n\geq 4$,
  $\widetilde{\mathbb{E}}_6$, $\widetilde{\mathbb{E}}_7$,
  $\widetilde{\mathbb{E}}_8$, or
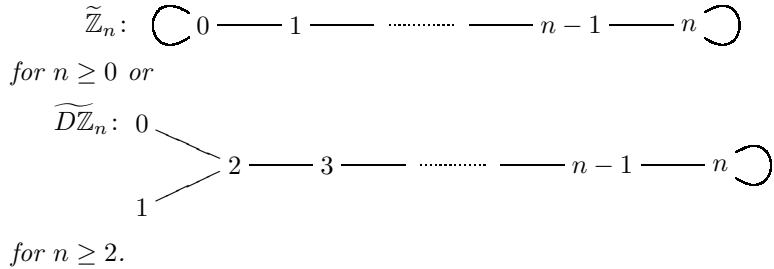
\[\widetilde{\mathbb{Z}}_n\colon \hskip 0.7cm
\xymatrix{
0\ar@{-}@(ul,dl) \ar@{-}[r] & 1\ar@{-}[r] & \ar@{..}[r] & \ar@{-}[r] &
n-1\ar@{-}[r] & n\ar@{-}@(ur,dr)
}\]
for $n\geq 0$ or 
\[\widetilde{D\mathbb{Z}}_n\colon 
\xymatrix@R=3pt{
0\ar@{-}[dr] & & & & & & \\
 & 2 \ar@{-}[r] & 3\ar@{-}[r] & \ar@{..}[r] &   
\ar@{-}[r] & n-1\ar@{-}[r] & n\ar@{-}@(ru,rd) \\
1\ar@{-}[ur] & & & & & &
}\] 
for $n\geq 2$. 
\item[(c)] If $\lambda < 2$, then the dimensions of the modules in a
  minimal projective resolution of any finitely generated
  $\L$-module is bounded.
\end{enumerate}
\end{thm}
The trichotomy in Theorem \ref{solberg-thm:benson} corresponds to the
division in wild, tame and finite representation type as pointed out
in \cite{B}. By \cite[Theorem 2.5]{EHSST} the
complexity of any finitely generated module over an algebra satisfying
\textbf{(Fg)} is bounded above by the Krull dimension of the
Hochschild cohomology ring, hence finite. It follows from this that a
weakly symmetric algebra with radical cube zero only can satisfy
\textbf{(Fg)} in case (b) and (c) in the above theorem. We remarked
above that all the algebras in (c) satisfy \textbf{(Fg)}, so we only
need to consider the algebras given in (b). 

The above result gives the quiver of the algebra $\L$, but since $\L$
is supposed to be weakly symmetric with $\mathfrak{r}^3=(0)$, it is
easy to write down the possible relations. In these relations one can
introduce scalars from the field. Most of the times the results are
independent of these scalars, except in the $\widetilde{\mathbb{Z}}_0$
case and the $\widetilde{\mathbb{A}}_n$ case, where it suffices to
introduce one scalar $q$ in one commutativity relation. The next
sections are devoted to discussing these cases.

We end this section with a remark which is a good guide in doing
actual computations.  Let $\L$ be a finite dimensional selfinjective
Koszul algebra over an algebraically closed field $k$ satisfying
\textbf{(Fg)}. In finding non-nilpotent generators for $Z_\gr(E(\L))$
as an algebra, then \cite[Proposition 4.4]{EHSST} and Theorem
\ref{thm:bgss} says essentially that the degree of a non-nilpotent
generator for $Z_\gr(E(\L))$ is a multiplum of the
$\Omega$-periodicity of some $\Omega$-periodic module. In actual
computations this usually gives a good indication where to look for
non-nilpotent generators.

%%%%%%%%%%%%%%%%%%%%%%%%%%%%%%%%%%%%%%%%%%%%%%%%%%%%%%%%%%%%%%%%%%%%

\section{The $\tilde{\mathbb{A}}_n$-case}
In this section we characterize when a weakly symmetric algebra over a
field $k$ with radical cube zero of type $\tilde{\mathbb{A}}_n$
satisfies \textbf{(Fg)}. For a computation of the Hochschild
cohomology ring of a more general class of algebras containing the
algebras we consider in this section consult \cite{ST,ST2}.

Let $Q$ be the quiver given by
\[
\xymatrix{ 0\ar@<1ex>[r]^-{a_0} &
  1\ar@<1ex>[l]^-{\overline{a}_0}\ar@<1ex>[r]^-{a_1} &
  2\ar@<1ex>[l]^-{\overline{a}_1}\ar@<1ex>[r]^-{a_2} &
  \ar@<1ex>[l]^-{\overline{a}_2}\ar@{..}[r] & \ar@<1ex>[r]^-{a_{n-2}}&
  n-1\ar@<1ex>[l]^-{\overline{a}_{n-2}}\ar@<1ex>[r]^-{a_{n-1}} &
  n\ar@<1ex>[l]^-{\overline{a}_{n-1}}\ar@<1ex>[r]^-{a_n} &
  0\ar@<1ex>[l]^-{\overline{a}_n} }\] 
where the extreme vertices are identified as the notation
suggests. Let $k$ be a field, and let $I'$ be the ideal in $kQ$
generated by the elements $\{ a_ia_{i+1}\}_{i=0}^n$,
$\{\overline{a}_{i+1}\overline{a}_i\}_{i=0}^n$, and $\{
a_i\overline{a}_i+q_i\overline{a}_{i-1}a_{i-1}\}_{i=0}^n$, for some
nonzero elements $q_i$ in $k$. Here we compute the indices modulo
$n+1$. By changing basis the algebra $\L=kQ/I'$ can be represented by
the same quiver, but an ideal $I$ generated by the elements $\{
a_ia_{i+1}\}_{i=0}^n$, $\{\overline{a}_{i+1}\overline{a}_i\}_{i=0}^n$, and
$\{ a_i\overline{a}_i+\overline{a}_{i-1}a_{i-1}\}_{i=1}^{n}\cup
\{a_0\overline{a}_0+q\overline{a}_na_n\}$ for some nonzero element $q$
in $k$. That is, $\L\simeq kQ/I$.

Next we apply Theorem \ref{thm:Fgcharacterization} to
characterize when $\L$ satisfies \textbf{(Fg)}.

\begin{prop}
Let $Q$ and $I$ in $kQ$ be as above for a field $k$. Then $\L=kQ/I$
satisfies \emph{\textbf{(Fg)}} if and only if $q$ is a root of unity.
\end{prop} 
\begin{proof}  
The graded centre of $E(\L)$ and $E(\L)^\op$ are the same and the
criterion for checking \textbf{(Fg)} can equivalently be performed
using $E(\L)^\op$.  The algebra $E=E(\L)^\op$ is given by
$kQ/\langle\{ a_i\overline{a}_i-\overline{a}_{i-1}a_{i-1}\}_{i=1}^{n},
qa_0\overline{a}_0-\overline{a}_na_n\rangle$. Consider the length
left-lexicographic ordering of the paths in $kQ$ by letting the
vertices be less than any arrow, and order the arrows like
$\overline{a}_0<\cdots <\overline{a}_n<a_0<\cdots<a_n$. Then $\{
a_i\overline{a}_i-\overline{a}_{i-1}a_{i-1}\}_{i=1}^{n},
qa_0\overline{a}_0-\overline{a}_na_n$ is a Gr\"obner basis for the
ideal they generate in $kQ$. In addition $E$ is bigraded via the
degrees in $a$-arrows and $\overline{a}$-arrows. The graded centre
inherits the bigrading from $E$, and the elements in $Z_\gr(E)$ are
sums of bi-homogeneous elements.

Let 
\[x=\sum_{i=0}^n a_ia_{i+1}\cdots a_na_0\cdots a_{i-1}\] 
and
\[y=\sum_{i=0}^n \overline{a}_{i-1}\overline{a}_{i-2}\cdots
\overline{a}_0\overline{a}_n\cdots \overline{a}_i\] in $E$. Any
path in $kQ$ viewed as an element in $E$ can be written uniquely
as $\overline{A}Ay^rx^s$ for some natural numbers $r$ and $s$, and
some path $\overline{A}$ in the arrows $\{\overline{a}_{i-1}\}$ and
some path $A$ in the arrows $\{a_i\}$, each of which has length at
most $n$.

We have that $a_ix=xa_i$, $\overline{a}_{i-1}y=y\overline{a}_{i-1}$,
$\overline{a}_{i-1}x=qx\overline{a}_{i-1}$ and $a_iy=q^{-1}ya_i$ for
all $i$. Hence, if $q$ is a $d$-th root of unity, then
$\{x^{2d},y^{2d}\}$ is in $Z_\gr(E)$.  It is immediate that $E$ is a
finitely generated module over $Z_\gr(E)$ when $q$ is a root of unity.

Suppose that $q$ is not a root of unity. Assume that $z$ is an element
in $Z_\gr(E)$ of homogeneous bi-degree $(r,s)$ with $r+s\geq 1$. Then
$z=z_0+z_1+\cdots+z_n$ with $z_i$ in $e_iEe_i$ of bi-degree $(r,s)$
for $i=0,1,\ldots, n$. If $z_i\neq 0$, then $z_ia_i = a_iz_{i+1}$ and
$\overline{a}_{i-1}z_i = z_{i-1}\overline{a}_{i-1}$ are non-zero,
hence both $z_{i-1}$ and $z_{i+1}$ are non-zero. It follows that each
$z_i$ is non-zero for all $i$. By considering some power of $z$ we can
assume that each $z_i=\alpha_ie_iy^{r'}x^{s'}e_i$ for some positive
integers $r'$ and $s'$ and $\alpha_i$ in $k\setminus\{0\}$. Since at
least one of $r'$ and $s'$ is positive, it follows without loss of
generality that $q^{r'(n+1)}=1$.  This is a contradiction. Hence, when
$q$ is not a root of unity, $Z_\gr(E)$ is $k$. Furthermore
\textbf{(Fg)} is not satisfied, and we have shown that $\L$ satisfies
\textbf{(Fg)} if and only if $q$ is a root of unity.
\end{proof}

%%%%%%%%%%%%%%%%%%%%%%%%%%%%%%%%%%%%%%%%%%%%%%%%%%%%%%%%%%%%%%%%%%%%

\section{The $\widetilde{\mathbb{Z}}_n$-case}
This section is devoted to proving that the weakly symmetric algebras
over a field $k$ with radical cube zero of type
$\widetilde{\mathbb{Z}}_n$ all satisfy \textbf{(Fg)} for $n>0$, again
using Theorem \ref{thm:Fgcharacterization}. The case $n=0$ is
discussed in Section \ref{section:qea}.

The quiver $Q$ of the algebras of type $\widetilde{\mathbb{Z}}_n$ are
given by
\[
\xymatrix{
0\ar@(ul,dl)_b \ar@<1ex>[r]^-{a_0} &
1\ar@<1ex>[l]^-{\overline{a}_0}\ar@<1ex>[r]^-{a_1} & 
2\ar@<1ex>[l]^-{\overline{a}_1}\ar@<1ex>[r]^-{a_2}  
& \ar@<1ex>[l]^-{\overline{a}_2}\ar@{..}[r] & \ar@<1ex>[r]^--{a_{n-2}}&
n-1\ar@<1ex>[l]^--{\overline{a}_{n-2}}\ar@<1ex>[r]^--{a_{n-1}} &
n\ar@<1ex>[l]^--{\overline{a}_{n-1}}\ar@(ur,dr)^c
}\]
where we impose the following relations when $n>0$ 
\begin{multline}
\{b^2+a_0\overline{a}_0,ba_0,\overline{a}_0b,\{a_ia_{i+1}\}_{i=0}^{n-2},
\{\overline{a}_i\overline{a}_{i-1}\}_{i=1}^{n-1},\notag\\
\{a_i\overline{a}_i+\overline{a}_{i-1}a_{i-1}\}_{i=1}^{n-1},
a_{n-1}c,c\overline{a}_{n-1},c^2+q\overline{a}_{n-1}a_{n-1}\}.\notag
\end{multline}
Let $I$ denote the ideal generated by this set of relations.

One could deform this algebra by introducing non-zero coefficients in
all the ``commutativity'' relations above. However, with a suitable
basis change, we can remove all these commutativity coefficients and
only have one remaining, for instance the $q$ as chosen above. In
addition, if $k$ contains a square root of $q$, then we can replace
$q$ by $1$.

\begin{prop}
Let $\L=kQ/I$, where $Q$ and $I$ are as above. Then $\L$ satisfies
\emph{\textbf{(Fg)}}. 
\end{prop}
\begin{proof}
We have that 
\[E(\L)^\op=kQ/(b^2-a_0\overline{a}_0,
\{a_i\overline{a}_i-\overline{a}_{i-1}a_{i-1}\}_{i=1}^{n-1}, 
qc^2-\overline{a}_{n-1}a_{n-1}).\] 
Let $x=\overline{a}_{n-1}a_{n-1}+ \sum_{i=0}^{n-1}
a_i\overline{a}_i$. Then direct calculations show that $x$ is in
$Z_\gr(E(\L))$. For $d$ in $\{b,c\}$, denote by $d[s]$ the shortest
closed path in $Q$, starting in vertex $s$ involving $d$. Then let
$y=\sum_{i=0}^n(b[i]c[i] + c[i]b[i])$. Straightforward computations
show that $y$ is in $Z_\gr(E(\L))$.

We use the above to show that $E(\L)$ is a finitely generated module
over $Z_\gr(E(\L))$ (or actually the subalgebra generated by $x$ and
$y$). Any path in $Q$, starting and ending in the same vertex and
involving only the arrows $a_?$ and $\overline{a}_?$, is as an element
in $E(\L)$ a power of $x$ times the appropriate idempotent. Then any
path $p$ in $Q$, as an element in $E(\L)$, can be written as $x^mp'$
for some path $p'$, where $p'=X^lr_X$ for $X$ in
$\{b[i]c[i],c[i]b[i]\}_{i=0}^n$ and $r_X$ is a proper subpath of $X$
from the left. Hence, any element in $E(\L)$ is a linear combination
of elements of the form $\{x^m(b[i]c[i])^lr_{B_i},
x^m(c[i]b[i])^lr_{C_i}\}_{i=0}^n$. We directly verify that $b[i]^2$
and $c[i]^2$ are in the span of $x^me_i$ for all $i$. Then by
induction it is easy to see that $(b[i]c[i])^l$ and $(c[i]b[i])^l$ are
in $k\langle x,y\rangle(b[i]c[i],c[i]b[i],e_i)$ for all $l$. It
follows from this that $E(\L)$ is a finitely generated
$Z_\gr(E(\L))$-module, and hence $\L$ satisfies \textbf{(Fg)}.
\end{proof}

%%%%%%%%%%%%%%%%%%%%%%%%%%%%%%%%%%%%%%%%%%%%%%%%%%%%%%%%%%%%%%%%%%%%%%

\section{The $\widetilde{D\mathbb{Z}}_n$-case}

This section is devoted to showing that a symmetric finite dimensional
algebra $\L$ over a field $k$ of type $\widetilde{D\mathbb{Z}}_n$
satisfies \textbf{(Fg)}.

Let $Q$ be the quiver given by 
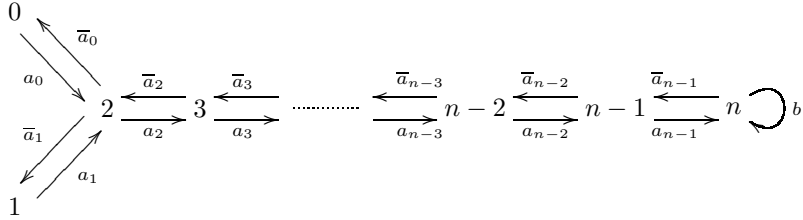
\[\xymatrix{%
0\ar@<-1ex>[dr]_{a_0} &  & & & & & & \\
 & 2
\ar@<-1ex>[ul]_{\overline{a}_0}\ar@<-1ex>[dl]_{\overline{a}_1}
\ar@<-1ex>[r]_{a_2} & 
               3\ar@<-1ex>[l]_{\overline{a}_2}\ar@<-1ex>[r]_{a_3} & 
               \ar@<-1ex>[l]_{\overline{a}_3} \ar@{..}[r] &   
\ar@<-1ex>[r]_{a_{n-3}} & n-2\ar@<-1ex>[l]_{\overline{a}_{n-3}}
\ar@<-1ex>[r]_{a_{n-2}} & 
 n-1\ar@<-1ex>[l]_{\overline{a}_{n-2}}\ar@<-1ex>[r]_{a_{n-1}} & 
n\ar@<-1ex>[l]_{\overline{a}_{n-1}}\ar@(ru,rd)^{b} \\
1\ar@<-1ex>[ur]_{a_1} & & & & & & & }
\]
Assume that $n>2$. Let $I$ be the ideal in $kQ$ generated by the
elements 
\begin{multline}
\{a_0\overline{a}_1, a_0a_2, a_1\overline{a}_0,
\overline{a}_2\overline{a}_1,
\overline{a}_0a_0-\overline{a}_1a_1,\overline{a}_2\overline{a}_0,
\overline{a}_1a_1-a_2\overline{a}_2,\notag\\ 
\{a_ia_{i+1}\}_{i=1}^{n-2},
\{\overline{a}_i\overline{a}_{i-1}\}_{i=2}^{n-1},
\{\overline{a}_{i-1}a_{i-1}+a_i\overline{a}_i\}_{i=3}^{n-1},\notag\\  
a_{n-1}b, \overline{a}_{n-1}a_{n-1}+qb^2, b\overline{a}_{n-1}\} 
\end{multline} 
for some $q$ in $k\setminus \{0\}$. When $n=2$, the ideal $I$ we
factor out is slightly different. This ideal is implicitly given in
the end of the proof of the next result. 

Before proving these algebras satisfies \textbf{(Fg)}, we discuss
deformations of the algebra $\L$. One could deform this algebra by
introducing non-zero coefficients in all the commutativity relations
above. However with a suitable basis change, we can remove all these
commutativity coefficients and only have one remaining, for instance
the $q$ as chosen above. We thank the referee for pointing out that
with a basis change given by $a_i\mapsto a_i$, $\overline{a}_i\mapsto
q\overline{a}_i$ and $b\mapsto b$, the scalar $q$ can be replaced by
$1$. 

\begin{prop}\label{prop:Z_n}
Let $Q$, $I$ and $\L$ be as above. Then $\L$ satisfies
\emph{\textbf{(Fg)}}.
\end{prop}
\begin{proof}
Suppose $n\geq 3$. Again we apply Theorem
\ref{thm:Fgcharacterization}. The opposite $E(\L)^\op$ of the Koszul
dual of $\L$ is given by $kQ$ modulo the relations generated by
\[\{a_0\overline{a}_0, a_1\overline{a}_1,
\overline{a}_0a_0+\overline{a}_1a_1+a_2\overline{a}_2,
\{\overline{a}_{i-1}a_{i-1}-a_i\overline{a}_i\}_{i=3}^{n-1},
\overline{a}_{n-1}a_{n-1} - b^2\}.\]

Let $\alpha=\overline{a}_0a_0$, $\beta=\overline{a}_1a_1$ and
$\gamma=a_2\overline{a}_2$. Note that
$\alpha\beta+\beta\alpha=\gamma^2$ and 
$\alpha\overline{a}_0=a_0\alpha=\beta\overline{a}_1=a_1\beta=0$.  Let
$x=\sum_{i=0}^nx_i$ with 
\[x_i=\begin{cases}
(\overline{a}_{n-1}a_{n-1})^2, & \text{for $i=n$},\\
(a_i\overline{a}_i)^2, & \text{for $2\leq i\leq n-1$},\\
a_0\beta \overline{a}_0, & \text{for $i=0$},\\
a_1\alpha\overline{a}_1, & \text{for $i=1$}.
\end{cases}
\]
Then direct computations show that $x$ is in $Z_\gr(E(\L))$.

For $\eta$ in $\{\alpha,\beta,b\}$ denote by $\eta[s]$ the shortest
path in $Q$ starting in vertex $s$ involving $\eta$, whenever this
makes sense. In particular,
$\alpha[s]=\overline{a}_{s-1}\overline{a}_{s-2}\cdots\overline{a}_2\alpha
a_2a_3\cdots a_{s-1}$ for $3\leq s\leq n$, and $\alpha[2]=\alpha$. We
leave it to the reader to write out the similar formulae for the other
cases.

Let $y'=\sum_{i=0}^n y_i'$ with
\[y_i'=\begin{cases}
b[0], & i=0,\\
b[1], & i=1,\\
b[i]\alpha[i]-\alpha[i]b[i], & 2\leq i\leq n \text{\ and $i$ even,}\\
b[i]\beta[i] -\alpha[i]b[i], & 2\leq i\leq n \text{\ and $i$ odd.} 
\end{cases}
\]
We want to show that $y=(y')^2$ is in $Z_\gr(E(\L))$. In doing so the
following equalities are useful to have, 
\begin{align}
\alpha\gamma & = -\alpha\beta  = \gamma\beta\tag{$\dagger$}\\
\beta\gamma  & = -\beta\alpha  = \gamma\alpha\tag{$\ddagger$}
\end{align}
and 
\[b[i]\alpha[i]-\alpha[i]b[i] = -(b[i]\beta[i]-\beta[i]b[i])\] 
for all $2\leq i \leq n$. Note that the last property is equivalent to
having $b[i]\gamma[i]=\gamma[i]b[i]$ for $2\leq i\leq n$. The most
cumbersome calculations involve vertex $n$, and here it is useful to
note that $\alpha[n]\gamma[n]$ is equal to $\gamma[n]\alpha[n]$ if $n$
is odd, while it is equal to $\gamma[n]\beta[n]$ for $n$
even. Pointing out that $\alpha[n]^2=0$ and $\beta[n]\alpha[n]=0$ when
$n$ is even and odd, respectively, we leave it to the reader to show
that $y$ is in $Z_\gr(E(\L))$.

Next we want to show that $E(\L)$ is a finitely generated module over
the subalgebra $Z_2$ generated by $\{x_2,y_2\}$. Let $E_2=e_2E(\L)^\op
e_2$. Then $E_2$ is generated as an $k$-algebra by
$\{\alpha,\gamma,b[2]\}$. Let $\mu$ be any non-zero monomial in
$\{\alpha,\gamma,b[2]\}$. We have that $\alpha\gamma = - \gamma^2 -
\gamma\alpha$ and $b[2]\gamma = \gamma b[2]$, so we can suppose that
all the $\gamma$'s in $\mu$ can be moved to the left (since
$\gamma^2=x_2$). Furthermore, we have that $b[2]^2 = \gamma^{2n-3}$
and $\alpha^2=0$, so that we can write $\mu=\pm\gamma^t\mu_1\mu_2\cdots
\mu_r$ with $\mu_i$ in $\{\alpha,b[2]\}$ for all $i$ with $\mu_i\neq
\mu_{i+1}$. We have that $x_2=\gamma^2$ and $y_2=(b[2]\alpha)^2 +
(\alpha b[2])^2 + \alpha\gamma^{2(n-1)}$.  Let $\A$ be the set of
monomials in $\{\alpha,\gamma,b[2]\}$ with at most six factors. It
is then easy to see that the factor $E_2/Z_2\A$ is zero, and hence
$E_2$ is a finitely generated $Z_2$-module. Any oriented cycle in $Q$
not going through the vertex $2$, can be written as a power of $x$
times one of a finite set of cycles. It follows from this that
\textbf{(Fg)} is satisfied for $\L$, when $n\geq 2$.

Let $n=2$. Then $E(\L)^\op$ is given by $kQ$ modulo the relations
$\{a_0\overline{a}_0,a_1\overline{a}_1,
\overline{a}_0a_0+b^2+\overline{a}_1a_1\}$. Let
$\alpha=\overline{a}_0a_0$, $\beta=\overline{a}_1a_1$ and
$\gamma=b^2$. Let $x_0=\gamma[0]$, $x_1=\gamma[1]$, and
$x_2=\gamma^2=-(\alpha\gamma+\gamma\alpha)=-(\beta\gamma+\gamma\beta)$. Let
$y_0=b\alpha b[0]$, $y_1=b\beta b[1]$ and $y_2=(b\alpha-\alpha
b)^2=(b\beta - \beta b)^2$. Then it is easy to see that $x=x_0+x_1+x_2$
and $y=y_0+y_1+y_2$ are in $Z_\gr(E(\L))$. Using that $x_2=-\alpha b^2
- b^2\alpha$ and $\alpha y_2=\alpha b\alpha b\alpha$, it is immediate
that $\L$ satisfies \textbf{(Fg)} also in this case.
\end{proof}
The reduction to showing that $e_2Ee_2$ is a finitely generated
$e_2Ze_2$-module as in the proof above will be used later again. 

This example also provides us with additional information on Betti
numbers of periodic modules. Considering the $\L$-module $M$ with
radical layers $\left(\begin{smallmatrix}
n\\ n\end{smallmatrix}\right)$ it is easy to see that $M$ is
$\Omega$-periodic with period $2n-1$, and all the projective modules
in an initial periodic minimal projective resolution are
indecomposable except projective number $n-1$, which is a direct sum
of two indecomposable projective modules. This gives an example of an 
$\Omega$-periodic module with non-constant Betti numbers. 

%%%%%%%%%%%%%%%%%%%%%%%%%%%%%%%%%%%%%%%%%%%%%%%%%%%%%%%%%%%%%%%%%%%%%

\section{The $\widetilde{\mathbb{D}}_n$-case}
This section is devoted to proving that the weakly symmetric algebras
over a field $k$ with radical cube zero of type
$\widetilde{\mathbb{D}}_n$ all satisfy \textbf{(Fg)}.

Let $Q$ be the quiver given by 
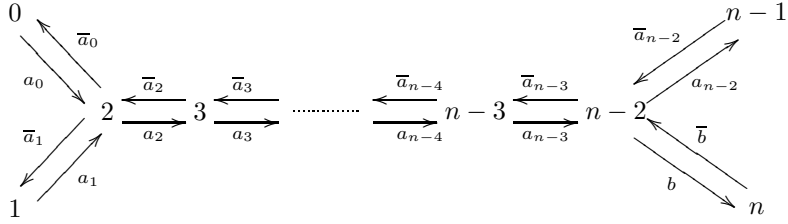
\[\xymatrix{%
0\ar@<-1ex>[dr]_{a_0} &  & & & & & & n-1\ar@<-1ex>[ld]_{\overline{a}_{n-2}}\\
 & 2
\ar@<-1ex>[ul]_{\overline{a}_0}\ar@<-1ex>[dl]_{\overline{a}_1}
\ar@<-1ex>[r]_{a_2} & 
               3\ar@<-1ex>[l]_{\overline{a}_2}\ar@<-1ex>[r]_{a_3} & 
               \ar@<-1ex>[l]_{\overline{a}_3} \ar@{..}[r] &   
\ar@<-1ex>[r]_{a_{n-4}} & n-3\ar@<-1ex>[l]_{\overline{a}_{n-4}}
\ar@<-1ex>[r]_{a_{n-3}} & 
 n-2\ar@<-1ex>[l]_{\overline{a}_{n-3}}\ar@<-1ex>[ur]_{a_{n-2}}
\ar@<-1ex>[rd]_b &  \\
1\ar@<-1ex>[ur]_{a_1} & & & & & & & n\ar@<-1ex>[ul]_{\overline{b}} }
\]
Assume that $n> 4$. Let $I$ be the ideal in $kQ$ generated by the
elements
\begin{multline}
\{a_0\overline{a}_1, a_0a_2, a_1\overline{a}_0,
\overline{a}_2\overline{a}_1,
\overline{a}_0a_0-\overline{a}_1a_1,\overline{a}_2\overline{a}_0,
\overline{a}_1a_1-a_2\overline{a}_2,\notag\\ 
\{a_ia_{i+1}\}_{i=1}^{n-3},
\{\overline{a}_i\overline{a}_{i-1}\}_{i=2}^{n-2},
\{\overline{a}_{i-1}a_{i-1}+a_i\overline{a}_i\}_{i=3}^{n-3},\notag\\  
\overline{a}_{n-2}b, \overline{b}a_{n-2}, a_{n-3}b,
a_{n-2}\overline{a}_{n-2} - b\overline{b}, b\overline{b} -
\overline{a}_{n-3}a_{n-3}\}. 
\end{multline} 
When $n=4$ then $I$ is generated by 
\begin{multline}
\{a_0\overline{a}_1, a_0a_2,a_0b, a_1\overline{a}_0,a_1a_2,a_1b, 
\overline{a}_2\overline{a}_0, \overline{a}_2\overline{a}_1,
\overline{a}_2b,\notag\\ \overline{b}\overline{a}_0,
\overline{b}\overline{a}_1, \overline{b}a_2, 
\overline{a}_0a_0-\overline{a}_1a_1, \overline{a}_1a_1-b\overline{b}, 
b\overline{b}-a_2\overline{a}_2\}.
\end{multline}
Similarly as before, deformations via coefficients in the
commutativity relations can be removed via an appropriate basis
change. Given this we can show the following. 
\begin{prop}\label{prop:Dn-tilde}
Let $Q$, $I$ and $\Lambda$ be as above. Then  $\Lambda$ satisfies
\emph{\textbf{(Fg)}}.  
\end{prop}
\begin{proof}
We apply again Theorem 1.3. The opposite algebra $E =
E(\Lambda)^{\op}$ of the Koszul dual of $\Lambda$ is given by $kQ$
modulo the relations generated by
\begin{multline}
\{a_0\overline{a}_0, a_1\overline{a}_1,
\overline{a}_0a_0 + \overline{a}_1a_1 + a_2\overline{a}_2,
\{\overline{a}_{i-1}a_{i-1}-a_i\overline{a}_i\}_{i=3}^{n-3},\notag\\  
\overline{a}_{n-2}a_{n-2}, \overline{b}b,
a_{n-2}\overline{a}_{n-2} + b\overline{b} + \overline{a}_{n-3}a_{n-3}\} 
\end{multline} 
when $n>4$. For $n=4$ the relations are given by 
\[\{a_0\overline{a}_0,
a_1\overline{a}_1,\overline{a}_2a_2,\overline{b}b,  
\overline{a}_0a_0 + \overline{a}_1a_1 + a_2\overline{a}_2+b\overline{b}
\}.\]
Let $\alpha = \overline{a}_0a_0, \beta = \overline{a}_1a_1$ and
$\gamma = a_2\overline{a}_2$.  Furthermore we write $\delta =
a_{n-2}\overline{a}_{n-2}$, $\omega = b\overline{b}$ and $\eta =
\overline{a}_{n-3}a_{n-3}$. Note that $\alpha\beta + \beta\alpha =
\gamma^2$ and $\alpha\overline{a}_0= a_0\alpha=
\beta\overline{a}_1=a_1\beta =0$ and as well $\delta\overline{a}_{n-2}
= a_{n-2}\delta = \omega\overline{b} = b \omega = 0$.  

Let $x=\sum_i x_i$, where
\begin{align}
x_0 & = a_0\beta\overline{a}_0, & x_{n-2} & = \eta^2, \notag\\
x_1 & = a_1\alpha\overline{a}_1, & x_{n-1} & =
\overline{a}_{n-2}\omega a_{n-2},\notag\\ 
x_i & = (a_i\overline{a}_i)^2, \text{for\ } 2\leq i < n-2, & x_n & =
\overline{b}\delta b.\notag 
\end{align}
Then direct computations show that $x$ is in $Z_\gr(E(\Lambda))$. 
This is also true when $n=4$. Note that in this case
we have $\alpha + \beta + \delta + \omega = 0$.

When $n=4$ we find another element of degree 4 in the centre of $E$,
namely $w = \sum_{i=0}^4 w_i$, where
\begin{align}
w_0 & = a_0\delta \overline{a}_0, & 
w_2 & = (\alpha + \delta)^2 = \alpha\delta + \delta\alpha, & w_4 & =
\overline{b}\beta b.\notag\\
w_1 & = a_1\omega \overline{a}_1, & w_3 & = \overline{a}_3\alpha a_3, &
& \notag
\end{align}
 
We assume now that $n>4$. Suppose $\rho$ is one of $\alpha, \beta,
\delta, \omega$.  As before, we write $\rho[s]$ for the shortest
closed path starting at $s$ which involves $\alpha$.

Let $y = \sum_{i=0}^n y_i$, where the $y_i$ are defined as follows.
For $2\leq r\leq n-2$,
\[y_r = \begin{cases}
  \alpha[r]\delta[r] + \delta[r]\alpha[r], & \ r, n \text{\ even},\notag\\
\alpha[r]\delta[r] + \omega[r]\beta[r], & \ r \text{\ odd}, n
\text{\ even},\notag\\ 
\alpha[r]\delta[r] - \omega[r]\alpha[r], & \ r \text{\ even}, n
\text{\ odd},\notag\\ 
\beta[r]\omega[r] - \omega[r]\alpha[r], & \ r, n \text{\ odd}.\notag
\end{cases}\]
Furthermore, 
\[y_0 = \delta[0],\quad y_1 = \omega[1],\quad y_{n-1} =
\alpha[n-1],\quad y_n = \beta[n].\] 
We want to show that $y$ is in the centre of $E(\Lambda)$. One way to
prove this is to first establish the following identities.  For $2\leq
r\leq n-2$,
\[\alpha[r](\delta[r]+ \omega[r]) = 
\begin{cases}
(\delta[r]+ \omega[r])\alpha[r], & n-r-1\text{\ even},\notag\\
(\delta[r]+ \omega[r])\beta[r],  & n-r-1 \text{\ odd}.\notag
\end{cases}\leqno{(1)}\]
Moreover
\[\delta[r](\alpha[r]+\beta[r]) =  
\begin{cases}
(\alpha[r]+ \beta[r])\delta[r], & r-1 \text{\ even},\notag\\
(\alpha[r]+ \beta[r])\omega[r], & r-1 \text{\ odd}.\notag
\end{cases}\leqno{(2)}\]
Similar formulae hold by interchanging $\alpha$ and $\beta$ in $(1)$, 
and by interchanging $\delta$ and $\omega$, in $(2)$. 
Using these formulae one gets several identities for the $y_r$.
Assume $2\leq r \leq n-2$. Then 
\[y_r = \begin{cases} 
\beta[r]\omega[r] + \omega[r]\beta[r],  & n, r \text{\ even},\notag\\
\delta[r]\alpha[r] + \beta[r]\omega[r], & r\text{\ odd}, n
\text{\ even},\notag\\ 
\delta[r]\alpha[r]- \beta[r]\delta[r] = \omega[r]\beta[r] -
\alpha[r]\omega[r], & n, r \text{\ odd},\notag\\
\delta[r]\alpha[r] - \alpha[r]\omega[r] = 
\omega[r]\beta[r] - \beta[r]\delta[r], & r\text{\ even}, n\text{\ odd}.
\end{cases}\] 
These show in particular that the anti-homomorphism induced by $a\to
\overline{a}$ and $\overline{a}\to a$ of $E^+$ fixes each $y_r$. This
means that one only has to check that $y$ commutes with the arrows
$a_r$, then it automatically commutes with $\overline{a}_r$.
Furthermore, one checks that $\alpha[r]\delta[r]a_r =
a_r\beta[r+1]\omega[r+1]$, and similarly $\beta[r]\omega[r]a_r =
a_r\alpha[r+1]\delta[r+1]$, for $2\leq r < n-2$.  Using all these
details, it is not difficult to check that $y$ commutes with all
$a_i$, and with $b$.

Next, we want to show that $E(\Lambda)$ is a finitely generated module
over the subalgebra generated by $\{ x, y\}$. As in the previous
section, it suffices to show that the local algebra $E_2 =
e_2E(\Lambda)e_2$ is finitely generated as a module over the
subalgebra $Z$ generated by $\{ x_2, y_2\}$.

Recall $x_2 = \gamma^2$; and we take
\[y_2 = \begin{cases}
\alpha\delta[2] + \delta[2]\alpha, & n \text{\ even}\notag\\
\delta[2]\alpha - \alpha\omega[2], & n \text{\ odd}.
\end{cases}\]
Note that this is also correct when $n=4$, so we can deal
with arbitrary $n\geq 4$ at the same time.

The algebra $E_2$ is generated by $\{\alpha, \gamma, \delta[2]\}$,
note that $\delta[2] + \omega[2] = - \gamma^{n-3}$. We have further
identities, namely
\[\gamma\alpha = -\gamma^2 - \alpha\gamma,\quad \delta[2]\gamma =
\gamma\omega[2].\]
One  checks that if  $n$ is even, $\delta[2]^2 = 0$, and that
for  $n$ odd, $\delta[2]\omega[2]=0$.

Let $\mathcal{A}$ be the set of monomials in $\{ \alpha, \delta[2],
\gamma\}$ with at most three factors. We want to show that $E_2/Z
\mathcal{A}$ is zero, hence that $E_2$ is finitely generated over
$Z$.

Assume $\mu \in E_2$ is a non-zero monomial in $\alpha, \delta[2]$ and
$\gamma$.  We can move all even powers of $\gamma$ to the left, note
that these lie in $Z$.  Furthermore, any factor of $\alpha$ in $\mu$
can be moved to the left, using $\gamma\alpha = -\alpha\gamma + z$ for
$z\in Z$, and also using that for $n$ even, $\delta[2] \alpha =
-\alpha\delta[2] + y_2$ and for $n$ odd, $\delta[2]\alpha = y_2 +
\alpha\omega[2] = y_2 - \alpha\delta[2] - \alpha z $ where $z =
\gamma^{n-3} \in Z$.  Hence we may assume none except possibly the
first factor of $\mu$ is equal to $\alpha$. Next, consider
submonomials of length three in $\delta[2], \gamma$ of $\mu$ where
successive factors are different.  If it is of the form
$\delta[2]\gamma\delta[2] = \delta[2]\omega[2]\gamma$ then is zero if
$n$ is odd, and if $n$ is even, it is equal to $-\delta[2]^2 -
\delta[2]\gamma^{n-3} = z\delta[2]\gamma$ with $z\in Z$.  Otherwise,
it is of the form $\gamma\delta[2]\gamma = \gamma^2\omega[2] = z
\delta[2] + z'\gamma^j$ with $z, z'\in Z$ and $j=0$ or $1$.  Using
these one shows by induction on the length of $\mu$ that $\mu$ belongs
to $Z\mathcal{A}$.
\end{proof}

%%%%%%%%%%%%%%%%%%%%%%%%%%%%%%%%%%%%%%%%%%%%%%%%%%%%%%%%%%%%%%%%%%%%%

\section{The $\widetilde{\mathbb{E}}_6$-case}
This section is devoted to showing that the weakly symmetric algebras
over a field $k$ with radical cube zero of type
$\widetilde{\mathbb{E}}_6$ satisfy \textbf{(Fg)}.

Let $Q$ be the quiver 
\[\xymatrix{
& & 4 \ar@<-1ex>[d]_{\overline{a}_3} & & \\
& & 3 \ar@<-1ex>[u]_{a_3}\ar@<-1ex>[d]_{\overline{a}_2} & & \\
0\ar@<-1ex>[r]_{a_0} &
1\ar@<-1ex>[r]_{a_1}\ar@<-1ex>[l]_{\overline{a}_0}  & 2 
\ar@<-1ex>[r]_{a_4}\ar@<-1ex>[u]_{a_2}\ar@<-1ex>[l]_{\overline{a}_1}   & 5
\ar@<-1ex>[l]_{\overline{a}_4}   \ar@<-1ex>[r]_{a_5} &
6\ar@<-1ex>[l]_{\overline{a}_5}}
\]
with relations 
\begin{multline}
\{\{a_ia_{i+1}\}_{i=0}^4, \{\overline{a}_i\overline{a}_{i-1}\}_{i=1}^5, 
\{\overline{a}_{i-1}a_{i-1}+a_i\overline{a}_i\}_{i=1,3,5},
a_1a_4,\notag\\ 
\overline{a}_2a_4, \overline{a}_4a_2, \overline{a}_4a_1,
\overline{a}_1a_1-a_2\overline{a}_2, a_2\overline{a}_2 - a_4\overline{a}_4\}.
\end{multline}
Let $\L=kQ/I$, where $I$ is the ideal generated by the relations given
above for a field $k$. As before, we could deform the algebra by
introducing non-zero scalars in the commutativity relations, but by a
suitable basis change all the scalars can be removed. Then we have the
following. 
\begin{prop}\label{prop:fgE_6tilde}
Let $Q$, $I$ and $\L$ be as above. Then $\L$ satisfies
\emph{\textbf{(Fg)}}.
\end{prop}
\begin{proof}
The opposite $E(\L)^\op$ of the Koszul dual of $\L$ is given by $kQ$
modulo the relations generated by 
\[\{a_0\overline{a}_0, \overline{a}_3a_3, \overline{a}_5a_5, 
\{\overline{a}_{i-1}a_{i-1} - a_i\overline{a}_i\}_{i=1,3,5},
\overline{a}_1a_1 + a_2\overline{a}_2 + a_4\overline{a}_4\}.\] Let
$\alpha=\overline{a}_1a_1$, $\beta=a_4\overline{a}_4$ and $\gamma=
a_2\overline{a}_2$.  Let $x_0=\gamma[0]$,
$x_1=(\alpha\gamma+\gamma\alpha)[1]$, and $x_2=\alpha^2\gamma +
\alpha\gamma\alpha + \gamma\alpha^2$. We want to define $x_i$ for
$i=3,4,5,6$ by symmetry. To do so, we need the following
details. Using the relations in $E(\L)$ it follows that
$a_1\alpha^2=\alpha^2\overline{a}_1=0$, and therefore
$\alpha^3=\beta^3=\gamma^3=0$. Applying this we infer that
\begin{align}
\alpha^2\gamma+\alpha\gamma\alpha+\gamma\alpha^2 & =
-(\alpha\gamma^2+\gamma\alpha\gamma+\gamma^2\alpha)\tag{$\dagger$}\\
\gamma^2\alpha & =-\gamma^2\beta\notag\\
\gamma\alpha\gamma & = - \gamma\beta\gamma\notag\\
\alpha\gamma^2 & = - \beta\gamma^2\notag
\end{align}
and similar formulae. We define now
$x_3=-(\alpha\gamma+\gamma\alpha)[3]$, $x_4=-\alpha[4]$,
$x_5=-(\alpha\beta+\beta\alpha)[5]$ and $x_6=-\alpha[6]$. Utilizing
symmetry direct computations then show that $x=\sum_{i=0}^6 x_i$ is in
$Z_\gr(E(\L))$. 

Computing $(\dagger)\cdot \gamma - \gamma\cdot (\dagger)$ we obtain
\[\gamma^2\alpha^2 - \alpha\gamma\alpha\gamma  =
 \alpha^2\gamma^2 - \gamma\alpha\gamma\alpha.\]
Furthermore, since $a_0\alpha^2=0=\alpha^2\overline{a}_0$, we have
that 
\[\gamma^2\alpha[1]+\alpha\gamma^2[1]=-\gamma\alpha\gamma[1] -
\alpha\gamma\alpha[1].\]
Using that
$a_0\alpha\gamma\alpha[1]=a_0\alpha\gamma^2[1]=0
=\alpha\gamma\alpha[1]\overline{a}_0=\gamma^2\alpha[1]\overline{a}_0$
and $y_2=\beta^2\gamma^2 - \gamma\beta\gamma\beta$, and letting 
\begin{align}
y_2 & = \gamma^2\alpha^2 - \alpha\gamma\alpha\gamma,\notag\\
y_1 & = - \gamma\alpha\gamma[1],\notag\\
y_0 & = \gamma^2[0],\notag
\end{align}
we obtain by symmetry an element $y=\sum_{i=0}^6 y_i$ in
$Z_\gr(E(\L))$. 

Next we show that $E(\L)$ is a finitely generated module over the
subalgebra generated by $\{x,y\}$. The algebra $e_2E(\L)e_2$ is
generated by $\{\alpha,\gamma\}$ as an algebra.  Given a monomial
$\mu$ in $\alpha$ and $\gamma$ we can use
$-\alpha\gamma^2=x_2+\gamma\alpha\gamma+\gamma^2\alpha$ to move the
occurrence of $\gamma^2$ to the left in $\mu$. Hence, except for a
short initial part, we can assume that $\mu$ is a word in $\{\alpha,
\alpha^2,\gamma\}$. If $\alpha\gamma\alpha^2$ occurs somewhere in
$\mu$, the equality $\gamma\alpha^2 = - x_2 - \alpha^2\gamma -
\alpha\gamma\alpha$ gives $\alpha\gamma\alpha^2= -\alpha x_2
-\alpha^2\gamma\alpha$ and the occurrence of $\alpha^2$ is moved
further to the left in creating one new monomial and one expression
$-x_2\mu'$, where $\mu'$ is a monomial of degree three less than
$\mu$. Hence, except for a short initial part, we can assume that
$\mu$ is a word in $\{\alpha,\gamma\}$. The equality $y_2 +
\gamma^2\alpha^2 = \alpha\gamma\alpha\gamma $ implies that $\gamma
y_2=\gamma\alpha\gamma\alpha\gamma$ and $y_2
\alpha=\alpha\gamma\alpha\gamma\alpha$. By induction we obtain that
any monomial in $\alpha$ and $\gamma$ can be written as a linear
combination of products of powers of $\{x_2,y_2\}$, and a finite set
of monomials in $\alpha$ and $\gamma$. Hence $e_2E(\L)e_2$ is a
finitely generated module over the algebra generated by
$\{x_2,y_2\}$. As before it follows from this that \textbf{(Fg)} holds
for $\L$.
\end{proof}

%%%%%%%%%%%%%%%%%%%%%%%%%%%%%%%%%%%%%%%%%%%%%%%%%%%%%%%%%%%%%%%%%%

\section{The $\widetilde{\mathbb{E}}_7$-case}
This section is devoted to proving that the weakly symmetric algebras
over a field $k$ with radical cube zero of type
$\widetilde{\mathbb{E}}_7$ satisfy \textbf{(Fg)}.

Let $Q$ be the quiver 
\[\xymatrix{ 
& & & 4 \ar@<-1ex>[d]_{\overline{a}_3} & & & \\
0\ar@<-1ex>[r]_{a_0} &
1\ar@<-1ex>[r]_{a_1}\ar@<-1ex>[l]_{\overline{a}_0} 
& 2\ar@<-1ex>[r]_{a_2}\ar@<-1ex>[l]_{\overline{a}_1} &
3\ar@<-1ex>[r]_{a_4}\ar@<-1ex>[u]_{a_3}\ar@<-1ex>[l]_{\overline{a}_2}  
& 5\ar@<-1ex>[r]_{a_5}\ar@<-1ex>[l]_{\overline{a}_4}
 & 6\ar@<-1ex>[r]_{a_6}\ar@<-1ex>[l]_{\overline{a}_5} &
7\ar@<-1ex>[l]_{\overline{a}_6}} 
\]
with relations 
\begin{multline}
\{\{a_ia_{i+1}\}_{i=0}^5, \{\overline{a}_i\overline{a}_{i-1}\}_{i=1}^6, 
\{\overline{a}_{i-1}a_{i-1}+a_i\overline{a}_i\}_{i=1,2,5,6},
a_2a_4,\notag\\ 
\overline{a}_3a_4, \overline{a}_4a_3, \overline{a}_4\overline{a}_2,
\overline{a}_2a_2-a_3\overline{a}_3, a_3\overline{a}_3 - a_4\overline{a}_4\}.
\end{multline}
Let $\L=kQ/I$, where $I$ is the ideal generated by the relations given
above for a field $k$. As for the $\widetilde{\mathbb{E}}_6$-case,
deforming the algebra by introducing non-zero scalars in the
commutativity relations does not change the algebra up to
isomorphism. Then we have the following. 
\begin{prop}\label{prop:fgE_7tilde}
Let $Q$, $I$ and $\L$ be as above. Then $\L$ satisfies
\emph{\textbf{(Fg)}}.
\end{prop}
\begin{proof}
The opposite $E(\L)^\op$ of the Koszul dual of $\L$ is given by $kQ$
modulo the relations generated by 
\[\{a_0\overline{a}_0, \overline{a}_3a_3, \overline{a}_6a_6, 
\{\overline{a}_{i-1}a_{i-1} - a_i\overline{a}_i\}_{i=1,2,5,6},
\overline{a}_2a_2 + a_3\overline{a}_3 + a_4\overline{a}_4\}.\]
Let $\alpha=\overline{a}_2a_2$, $\beta=a_4\overline{a}_4$ and
$\gamma=a_3\overline{a}_3$. One easily shows that
$\alpha^4=\beta^4=\gamma^2=0$. Furthermore using that
$\alpha+\beta+\gamma = 0$,  
$\beta\gamma = -\alpha\gamma$ and $\gamma\beta=-\gamma\alpha$, we
get 
\begin{align}
\beta^3\gamma + \beta^2\gamma\beta + \beta\gamma\beta^2 + \gamma\beta^3 
 & = -[(\beta\gamma)^2 + \gamma\beta^2\gamma +
  (\gamma\beta)^2]\tag{$\dagger$}\\ 
& = - (\beta\gamma + \gamma\beta)^2 = - (\alpha\gamma +
\gamma\alpha)^2\notag\\  
& = -[(\alpha\gamma)^2 + \gamma\alpha^2\gamma + (\gamma\alpha)^2]\notag\\
& = \alpha^3\gamma + \alpha^2\gamma\alpha +
\alpha\gamma\alpha^2+\gamma\alpha^3.\notag
\end{align}

Let $x=\sum_{i=0}^7 x_i$ be the element of degree $8$ defined as follows
\begin{align}
x_0 & = \gamma[0], &  x_7 & = \gamma[7],\notag\\
x_1 & = (\alpha\gamma + \gamma\alpha)[1], &
x_6 & = (\beta\gamma + \gamma\beta)[6],\notag\\
x_2 & = (\alpha^2\gamma + \alpha\gamma\alpha + \gamma\alpha^2)[2], & 
x_5 & = (\beta^2\gamma  + \beta\gamma\beta+ \gamma\beta^2)[5],\notag\\
x_3 & = \alpha^3\gamma + \alpha^2\gamma\alpha + \alpha\gamma\alpha^2 +
\gamma \alpha^3, & x_4 & = -(\alpha\gamma\alpha)[4].\notag
\end{align}
Using $(\dagger)$, symmetry in $\alpha$ and $\beta$ and preforming
straightforward computations, we infer that $x$ is in $Z_\gr(E(\L))$.

Define $y = \sum_{i=0}^7 y_i$ as the following degree $12$ element in
$E(\L)$ with  
\begin{align}
y_0 & = -\gamma\alpha\gamma[0], & y_7 & = -\gamma\beta\gamma[7],\notag\\
y_1 & = \gamma\alpha^2\gamma[1], & y_6 & = \gamma\beta^2\gamma[6],\notag\\
y_2 & = (-\alpha^2\gamma\alpha^2 - \alpha^2\gamma\alpha\gamma +
\gamma\alpha^2\gamma\alpha)[2], & y_5 & = (-\beta^2\gamma\beta^2
-\beta^2\gamma\beta\gamma + \gamma\beta^2\gamma\beta)[5],\notag\\
y_3 & = \alpha^2\gamma\alpha^2\gamma + \gamma\alpha^2\gamma\alpha^2,
& y_4 & = \alpha^2\gamma\alpha^2[4].\notag
\end{align}
Using the last equality in $(\dagger)$ and $a_0\overline{a}_0=0$, it
follows that that $y_0$ and $y_1$ ``commute'' with $a_0$ and
$\overline{a}_0$.  Premultiplying the last equality in $(\dagger)$ with
$\alpha^2$ gives
\[\alpha^3\gamma\alpha^2 + \alpha^2\gamma\alpha^3 
= -[\alpha^2(\alpha\gamma)^2 + \alpha^2(\gamma\alpha)^2 +
  \alpha^2\gamma\alpha^2\gamma].\]
In computing $a_2y_3 - y_2a_2$ we obtain 
\[a_2y_3 - y_2a_2 = a_2[\alpha^2\gamma\alpha^2\gamma +
  \alpha^2\gamma\alpha\gamma\alpha + \alpha^2\gamma\alpha^3].\]
Furthermore, substitute for $\alpha^2\gamma\alpha^3$ using the above
expression, we can then cancel four terms and are left with
\[a_2y_3 - y_2a_2 = a_2[-\alpha^3\gamma\alpha^2 -
  \alpha^2(\alpha\gamma)^2]= 0,\] 
since $a_2\alpha^3=0$. Similar arguments give that $y_2$ and $y_3$
commute with $\overline{a}_2$ and $y_1$ and $y_2$ commute with $a_1$
and $\overline{a}_1$. One easily checks that $y_3$ and $y_4$ commute
with $a_3$ and $\overline{a}_3$. Utilizing that
$\gamma\beta(\gamma\beta + \beta\gamma)\beta\gamma = 0$, we conclude
by a direct substitution that $y_3=\beta^2\gamma\beta^2\gamma +
\gamma\beta^2\gamma\beta^2$. By symmetry in $\alpha$ and $\beta$ the
elements $\{y_3,y_5,y_6,y_7\}$ satisfy the required equations, so that
$y$ is an element in $Z_\gr(E(\L))$.

Now we show that $E(\L)$ is a finitely generated module over the
algebra generated by $\{x,y\}$. As before, we show that
$E_3=e_3E(\L)e_3$ is a finitely generated module over the subalgebra
generated by $\{x_3,y_3\}$. Let $\mu$ be any monomial in $\alpha$ and
$\gamma$. Recall that $x_3=\alpha^3\gamma + \alpha^2\gamma\alpha +
\alpha\gamma\alpha^2 + \gamma\alpha^3$. Using this equation we can
move the occurrence of $\alpha^3$ to the left, so that it remains to
analyze monomials $\mu$, where $\alpha^3$ does not occur except for in a
short initial part. Recall that $y_3 = \alpha^2\gamma\alpha^2\gamma +
\gamma\alpha^2\gamma\alpha^2$.  This gives that
$\gamma\alpha^2\gamma\alpha^2\gamma = \gamma y_3$. Hence we only need
to deal with submonomials of the form
$\gamma\alpha\gamma\alpha\gamma$, $\gamma\alpha\gamma\alpha^2\gamma$
and $\gamma\alpha^2\gamma\alpha\gamma$. Since $\gamma x_3\gamma =0$,
we obtain that $\gamma\alpha\gamma\alpha^2\gamma =
-\gamma\alpha^2\gamma\alpha\gamma$. In this way we can move
occurrences of $\alpha^2$ to the left, and if we create a submonomial
of the form $\gamma\alpha^2\gamma\alpha^2\gamma$, we replace it by
$\gamma y_3$ as above. Then it remains to analyze a monomial $\mu$,
which except for a short initial and a short terminal part, is of the
form $\gamma(\alpha\gamma)^t$ for some positive integer $t$. Recall
that $x_3=-\alpha\gamma\alpha\gamma
-\gamma\alpha\gamma\alpha-\gamma\alpha^2\gamma$, so that $\gamma x_3 =
-\gamma\alpha\gamma\alpha\gamma$. Combining all the observations
above, we have shown that any monomial in $\alpha$ and $\gamma$ can be
written as a linear combination of powers of $\{x_3,y_3\}$ times some
finite set of monomials in $\{\alpha,\gamma\}$. This shows that
$E(\L)$ is a finitely generated module over the subalgebra generated
by $\{x_3,y_3\}$, and therefore $\L$ satisfies \textbf{(Fg)}.
\end{proof}

%%%%%%%%%%%%%%%%%%%%%%%%%%%%%%%%%%%%%%%%%%%%%%%%%%%%%%%%%%%%%%%%%%%

\section{The $\widetilde{\mathbb{E}}_8$-case}
This section is devoted to showing that the weakly symmetric algebras
over a field $k$ with radical cube zero of type
$\widetilde{\mathbb{E}}_8$ satisfy \textbf{(Fg)}.

Let $Q$ be the quiver 
\[\xymatrix{
& & 3\ar@<-1ex>[d]_{\overline{a}_2} & & & & & \\
0\ar@<-1ex>[r]_{a_0} & 1\ar@<-1ex>[r]_{a_1}\ar@<-1ex>[l]_{\overline{a}_0} &
2\ar@<-1ex>[r]_{a_3}\ar@<-1ex>[u]_{a_2}\ar@<-1ex>[l]_{\overline{a}_1}   &
4\ar@<-1ex>[r]_{a_4}\ar@<-1ex>[l]_{\overline{a}_3} &
5\ar@<-1ex>[r]_{a_5}\ar@<-1ex>[l]_{\overline{a}_4}   & 
6\ar@<-1ex>[r]_{a_6}\ar@<-1ex>[l]_{\overline{a}_5} &
7\ar@<-1ex>[r]_{a_7}\ar@<-1ex>[l]_{\overline{a}_6}  &
8\ar@<-1ex>[l]_{\overline{a}_7} 
}\]
with relations 
\begin{multline}
\{\{a_ia_{i+1}\}_{i=0}^6, \{\overline{a}_i\overline{a}_{i-1}\}_{i=1}^7, 
\{\overline{a}_{i-1}a_{i-1}+a_i\overline{a}_i\}_{i=1,4,5,6,7},
a_1a_3,\notag\\ 
\overline{a}_2a_3, \overline{a}_3a_2, \overline{a}_3\overline{a}_1,
\overline{a}_1a_1-a_2\overline{a}_2, a_2\overline{a}_2 - a_3\overline{a}_3\}.
\end{multline}
Let $\L=kQ/I$, where $I$ is the ideal generated by the relations given
above for a field $k$. Deforming the algebra by introducing non-zero
scalars in the commutativity relations does not change the algebra up
to isomorphism. Given this, we have the following. 
\begin{prop}\label{prop:fgE_8tilde}
Let $Q$, $I$ and $\L$ be as above. Then $\L$ satisfies
\emph{\textbf{(Fg)}}. 
\end{prop}
\begin{proof}
The opposite $E(\L)^\op$ of the Koszul dual of $\L$ is given by $kQ$
modulo the relations generated by 
\[\{a_0\overline{a}_0, \overline{a}_2a_2, \overline{a}_7a_7, 
\{\overline{a}_{i-1}a_{i-1} - a_i\overline{a}_i\}_{i=1,4,5,6,7},
\overline{a}_1a_1 + a_2\overline{a}_2 + a_3\overline{a}_3\}.\]
Let $\alpha=\overline{a}_1a_1$, $\beta=a_3\overline{a}_3$ and
$\gamma=a_2\overline{a}_2$. As before let $E_2=e_2E(\L)^\op e_2$ be the
local algebra at vertex $2$. 

The structure of the proof is as before, first we exhibit two elements
$x$ and $y$ in $Z_\gr(E(\L))$. Then we show that $E(\L)$ is a finitely
generated module over the subalgebra generated by $\{x,y\}$, through
analyzing the the local algebra $E_2$.  The algebra $E_2$ is generated
by $\{\alpha, \gamma\}$, and they satisfy $\alpha^3=0$ and
$\gamma^2=0$.  Furthermore we have $\beta^6=0$.  This is equivalent
with
\begin{multline}
0= (\gamma\alpha^2\gamma\alpha^2 + \alpha^2\gamma\alpha^2\gamma)
+ (\alpha\gamma\alpha\gamma\alpha^2 +
\alpha^2\gamma\alpha\gamma\alpha)\tag{$\dagger$}\\ 
+ (\gamma\alpha^2\gamma\alpha\gamma + \gamma\alpha\gamma\alpha^2\gamma)
+ \alpha\gamma\alpha^2\gamma\alpha + (\gamma\alpha)^3 +
(\alpha\gamma)^3.\notag
\end{multline}
This will be used often. Furthermore, we have, from expanding $-\beta^3
= (\alpha+\gamma)^3$ that
\begin{equation}
-\beta^3 + \gamma\beta\gamma = -\beta^3 - \gamma\alpha\gamma  =
(\gamma\alpha^2 + \alpha\gamma\alpha +
\alpha^2\gamma).\tag{$\ddagger$}
\end{equation} 
Denote this element by $\rho$. Note that $\rho$ commutes with $\alpha$. 

In giving elements in $Z_\gr(E(\L))$, the following identity is
helpful. Let $\zeta = \rho^2$, then $\zeta$ has degree 12, and it lies
in the centre of $E_2$. To this end, first note that it can be written
in different ways.
\begin{align}
\zeta & = \gamma\alpha^2\gamma\alpha^2 +
  \alpha^2\gamma\alpha^2\gamma +  
\alpha^2\gamma\alpha\gamma\alpha 
+ \alpha\gamma\alpha^2\gamma\alpha +
\alpha\gamma\alpha\gamma\alpha^2\tag{$\bigtriangleup$}\\ 
& = -(\gamma\alpha^2\gamma\alpha\gamma +
\gamma\alpha\gamma\alpha^2\gamma + (\gamma\alpha)^3 +
(\alpha\gamma)^3) \cr 
& = (\gamma\alpha^2 + \alpha\gamma\alpha + \alpha^2\gamma)^2.\notag
\end{align}
By the second identity in $(\bigtriangleup)$ we have $\zeta \gamma = -
(\gamma\alpha)^3\gamma = \gamma\zeta$; and since $\rho$ commutes with
$\alpha$, so does $\zeta$. Hence $\zeta$ is in the centre of $E_2$.

\subsection*{An element of degree $12$ in $Z_\gr(E(\L))$} 
Define $x = \sum_{i=0}^8 x_i$, where  
\begin{align}
x_0 & = \gamma\alpha^2\gamma[0], & x_3 & = -(\alpha\gamma)^2\alpha[3], &  
x_6 & = \textstyle{\sum_{i=0}^2\beta^{2-i}\gamma\beta^i[6]},\notag\\
x_1 & = (\alpha\gamma\alpha^2\gamma + \gamma\alpha^2\gamma +
(\alpha\gamma)^2\alpha)[1], & x_4 & =  \textstyle{\sum_{i=0}^4
\beta^{4-i}\gamma\beta^i[4]}, & x_7 & = (\beta\gamma +
\gamma\beta)[7],\notag\\ 
x_2 & = \zeta, & x_5 & = \textstyle{\sum_{i=0}^3\beta^{3-i}\gamma\beta^i[5]},
& x_8 & = \gamma[8].\notag
\end{align}
We claim that $x$ is in the centre of $E(\L)$. On the branch of the
quiver starting with vertex $0$ one uses the first expression for
$\zeta$ given in $(\bigtriangleup)$. On the branch with vertex $3$ one
uses the second expression for $\zeta$. For the long branch, we need
$\zeta$ in terms of $\beta$ and $\gamma$.  Namely, we have
\[\zeta = (-\beta^3 + \gamma\beta\gamma)^2 = 
(\beta^5\gamma + \beta^4\gamma\beta + \beta^3\gamma\beta^2 +\cdots +
\gamma\beta^5).\]
To see this, write the RHS   as
\[\beta^3[\beta^2\gamma + \beta\gamma\beta + \gamma\beta^2] 
+ [\beta^2\gamma + \beta\gamma\beta + \gamma\beta^2]\beta^3.\] 
Expanding $\alpha^3=0$ gives $\beta^2\gamma + \beta\gamma\beta +
\gamma\beta^2 = -\beta^3 - \gamma\beta\gamma$, substitute this and use
$\beta^6=0$.

\subsection*{An element of degree $20$ in $Z_\gr(E(\L))$} 
Next we find an element $y$ of degree $20$ in $Z_\gr(E(\L))$. Define 
\[\omega = \alpha^2\gamma\alpha\gamma + \alpha\gamma\alpha\gamma\alpha
+ \gamma\alpha\gamma\alpha^2.\] 
This element commutes with $\alpha$. Furthermore, $\omega^2$ commutes
with $\gamma$, using the following which is easy to check.
\[\gamma\omega + \omega\gamma = -\zeta.\leqno{(\Box)}\]
Define $y=\sum_{i=0}^8 y_i$, where 
\begin{align}
y_0 & = \gamma\alpha\gamma\alpha^2\gamma\alpha\gamma[0], & y_5 & =
\textstyle{\sum_{i=0}^3 \beta^{3-i}\alpha^2\gamma\alpha^2\beta^i[5]}, \notag\\
y_1 & = (\alpha(\gamma\alpha\gamma\alpha^2\gamma\alpha\gamma) + 
(\gamma\alpha\gamma\alpha^2\gamma\alpha\gamma)\alpha +
(\alpha\gamma)^4\alpha)[1], & y_6 & =
\textstyle{\sum_{i=0}^2\beta^{2-i}\alpha^2\gamma\alpha^2\beta^i[6]}, \notag\\ 
y_2 & = \omega^2, & y_7 & = \textstyle{\sum_{i=0}^1
\beta^{1-i}\alpha^2\gamma\alpha^2\beta^i[7]}, \notag\\
y_3 & = (\alpha^2\gamma\alpha\gamma\alpha^2\gamma\alpha 
+ \alpha\gamma\alpha^2\gamma\alpha^2\gamma\alpha 
+ \alpha\gamma\alpha^2\gamma\alpha\gamma\alpha^2)[3], & y_8 & =
\alpha^2\gamma\alpha^2[8].\notag\\
y_4 & = \textstyle{\sum_{i=0}^4 \beta^{4-i}\alpha^2\gamma\alpha^2\beta^i +
2(\beta^4\gamma\beta^4)[4]}, & & \notag 
\end{align}
We claim that $y$ is an element in the centre of $E(\L)$. It is
straightforward to check that on each branch, away from the branch
vertex, we have $a_iy=ya_i$ and $\overline{a}_iy=y\overline{a}_i$, and
similarly $a_1y=ya_1$ and $\overline{a}_1y=y\overline{a}_1$.  The
remaining identity at the long branch will follow directly if we show
\begin{equation}
\omega^2 = \sum_{i=1}^5 \beta^{5-i} \alpha^2\gamma\alpha^2\beta^i +
2[\beta^5\gamma\beta^4 + \beta^4\gamma\beta^5].\tag{A}\label{eq:A}
\end{equation}
Furthermore, the remaining identity at vertex $3$ will follow directly from 
\begin{multline}
\omega^2 = \gamma[
\alpha^2\gamma\alpha\gamma\alpha^2\gamma\alpha 
+ \alpha\gamma\alpha^2\gamma\alpha^2\gamma\alpha 
+ \alpha\gamma\alpha^2\gamma\alpha\gamma\alpha^2]\notag\\
+ [(\alpha^2\gamma\alpha\gamma\alpha^2\gamma\alpha 
+ \alpha\gamma\alpha^2\gamma\alpha^2\gamma\alpha 
+ \alpha\gamma\alpha^2\gamma\alpha\gamma\alpha^2]\gamma.\tag{B}\label{eq:B}
\end{multline}
We use the following identity
\[(\alpha\gamma)^3\alpha^2\gamma\alpha 
+\alpha\gamma\alpha^2(\gamma\alpha)^3 + 
\alpha\gamma\alpha\gamma\alpha^2\gamma\alpha\gamma\alpha = 0,
\leqno{(\pentagon)}\]
which is obtained from $(\dagger)$ by premultiplying with
$\alpha\gamma$ and postmultiplying with $\gamma\alpha$.

We start with proving \eqref{eq:A}. First we calculate
\[\eqref{eq:A}\textrm{-I} = \sum_{i=1}^5 \beta^{5-i}
\alpha^2\gamma\alpha^2\beta^i.\]  
This can be written as $\beta^3m + m\beta^3$, where 
\begin{align} 
m & =\alpha\gamma\alpha^2\gamma\alpha^2 + \gamma\alpha^2\gamma\alpha^2\gamma +
  \alpha^2\gamma\alpha^2\gamma\alpha\notag\\
& = \alpha\zeta + \gamma\alpha^2\gamma\alpha^2\gamma -
  \alpha^2\gamma\alpha\gamma\alpha^2.\notag
\end{align}
Recall from $(\ddagger)$ that $\beta^3 = - \rho - \gamma\alpha\gamma$,
where $\rho^2=\zeta$. Substituting this gives
\begin{align} 
\beta^3m + m\beta^3 & = (-\rho m - \gamma\alpha\gamma m -
m\gamma\alpha\gamma - m \rho)\notag\\ 
& = -\rho m - (\gamma\alpha)^2\zeta + (\gamma\alpha\gamma\alpha^2)^2 
- (\alpha\gamma)^2\zeta + (\alpha^2\gamma\alpha\gamma)^2 -
m\rho.\notag  
\end{align}
Note that the two squares occur in $\omega^2$.  Next
\begin{align}
-\rho m & = -\rho \alpha\zeta - \rho\gamma\alpha^2\gamma\alpha^2\gamma
+ \rho \alpha^2\gamma\alpha \gamma\alpha^2 \notag\\
& = (\alpha\gamma)^4\alpha^2 + \alpha^2(\gamma\alpha)^4 
-  \rho\gamma\alpha^2\gamma\alpha^2\gamma
+ \rho \alpha^2\gamma\alpha \gamma\alpha^2\notag
\end{align}
(and reversing each term gives an identity for $-m\rho$). Note that
the first two terms occur in $\omega^2$. In the expression for
$\beta^2 m + m\beta^3 -\omega^2$ we can cancel two of the terms
immediately. Namely, we obtain that 
\[\rho\alpha^2\gamma\alpha\gamma\alpha^2 +
\alpha^2\gamma\alpha\gamma\alpha^2\rho = 0,\]
by first substituting $\rho\alpha^2 = \alpha^2\gamma\alpha^2 =
\alpha^2\rho$ and then pre- and post-multiply $(\dagger)$ with
$\alpha^2$. 

\subsubsection*{\eqref{eq:A}\emph{-II}}
We get from this that $\beta^3m + m\beta^3  - \omega^2$ is equal to
\[
(\alpha\gamma\alpha\gamma\alpha)^2 + 
(\alpha\gamma)^4\alpha^2 + \alpha^2(\gamma\alpha)^4 
- (\gamma\alpha)^2\zeta - (\alpha\gamma)^2\zeta
- \rho(\gamma\alpha^2\gamma\alpha^2\gamma)
 - (\gamma\alpha^2\gamma\alpha^2\gamma)\rho.\]

\subsubsection*{\eqref{eq:A}\emph{-III}}  By definition and an obvious
substitution 
\begin{align}  
\beta^5\gamma\beta^4 + \beta^4\gamma\beta^5 & = \beta^4[\beta\gamma +
  \gamma\beta]\beta^4 & = & \beta^4(\alpha^2 - \beta^2)\beta^4 & = &
\beta^4\alpha^2\beta^4\notag\\ & = \beta^3\gamma\alpha^2\gamma\beta^3
& = & (\rho + \gamma\alpha\gamma)
\gamma\alpha^2\gamma(\gamma\alpha\gamma + \rho) & = &
\rho(\gamma\alpha^2\gamma)\rho.\notag
\end{align}
So we must show that 
\[\eqref{eq:A}\textrm{-II} = - 2\rho\gamma\alpha^2\gamma\rho.
\leqno{(*)}\]  
Using that $\gamma\alpha^2\gamma = \gamma(\rho -
\alpha\gamma\alpha)=(\rho-\alpha\gamma\alpha)\gamma$ and
$\rho^2=\zeta$, we infer that  
$\rho(\gamma\alpha^2\gamma\alpha\gamma\alpha) - \zeta(\gamma\alpha)^2
= -\rho(\alpha\gamma)^3\alpha$ and
$\alpha\gamma\alpha\gamma\alpha^2\gamma\rho - (\alpha\gamma)^2\zeta =
-(\alpha\gamma)^3\alpha\rho$. 
Using the same identity as above we get that 
\[\alpha^2(\alpha\gamma)^4 -
\rho(\alpha\gamma)^3\alpha = - \alpha\gamma\alpha^2(\gamma\alpha)^3.\] 
The identity obtained from this by reversing the order in each monomial 
holds similarly. We add the appropriate equations and cancel, and we get 
\[\alpha^2(\gamma\alpha)^4 - \zeta(\gamma\alpha)^2 =
-\alpha\gamma\alpha^2(\gamma\alpha)^3 -
\rho(\gamma\alpha^2\gamma\alpha\gamma\alpha).\] 
The identity obtained by reversing the order in each term also
holds. We substitute these into \eqref{eq:A}\textrm{-II} and obtain
that it is equal to
\begin{multline}
-(\alpha\gamma\alpha\gamma\alpha)^2
-\alpha\gamma\alpha^2(\gamma\alpha)^3 -
(\alpha\gamma)^3\alpha^2\gamma\alpha  
\notag\\
- \rho(\gamma\alpha^2\gamma\alpha\gamma\alpha)
-(\alpha\gamma\alpha\gamma\alpha^2\gamma)\rho 
- \rho(\gamma\alpha^2\gamma\alpha^2\gamma)
-(\gamma\alpha^2\gamma\alpha^2\gamma)\rho. 
\end{multline}
The sum of the first three terms is zero, by $(\pentagon)$. 
Now we note that
\[\rho\gamma\alpha^2\gamma \rho =
\rho\gamma\alpha^2\gamma\alpha^2\gamma +
\rho\gamma\alpha^2\gamma\alpha\gamma\alpha, 
\leqno{(*)}\]
so we can replace two terms in \eqref{eq:A}\textrm{-II} by
$-\rho\gamma\alpha^2\gamma\rho$. The identity obtained from $(*)$ by
reversing the order also holds, and we can therefore replace the other
two terms in \eqref{eq:A}\textrm{-II}, which involve $\rho$ by
$-\rho\gamma\alpha^2\gamma\rho$.  So in total we get that 
\eqref{eq:A}\textrm{-II} is equal to $-2\rho\gamma\alpha^2\gamma\rho$
as required. This proves (A).

We prove now (B), that is
\begin{multline} 
\gamma\alpha\gamma\alpha^2\gamma\alpha^2\gamma\alpha 
+ \gamma\alpha\gamma\alpha^2\gamma\alpha\gamma\alpha^2
+ \alpha\gamma\alpha^2\gamma\alpha^2\gamma\alpha\gamma \notag\\
+ \alpha\gamma\alpha^2\gamma\alpha\gamma\alpha^2\gamma
- \alpha\gamma\alpha\gamma\alpha^2\gamma\alpha\gamma\alpha -
(\alpha\gamma)^4\alpha^2 - \alpha^2(\gamma\alpha)^4=0.\tag{$\Diamond$} 
\end{multline}
Take relation $(\dagger)$, and pre- and post-multiply it with $\alpha\gamma$,
this gives  
\[\alpha\gamma\alpha^2\gamma\alpha^2\gamma\alpha\gamma 
+ \alpha\gamma\alpha^2\gamma\alpha\gamma\alpha^2\gamma = 
- \alpha\gamma\alpha\gamma\alpha^2\gamma\alpha^2\gamma -
(\alpha\gamma)^5, 
\leqno{\eqref{eq:B}\textrm{-I}}\]
and we can replace terms 3 and 4 in $(\Diamond)$ by
\eqref{eq:B}\textrm{-I}. Symmetrically we can replace terms 1 and 2 in
$(\Diamond)$ by
\[-\gamma\alpha^2\gamma\alpha^2\gamma\alpha\gamma\alpha -
(\gamma\alpha)^5.\leqno{\eqref{eq:B}\textrm{-I$^*$}}\]
Next we claim that \eqref{eq:B}\textrm{-I} is equal to
\[(\alpha\gamma)^3\alpha^2\gamma\alpha + (\alpha\gamma)^4\alpha^2 +
\alpha\gamma\alpha\gamma\alpha^2\gamma\alpha\gamma\alpha. 
\leqno{\eqref{eq:B}\textrm{-II}}\]
Namely 
\begin{align}
\alpha\gamma\alpha\gamma\alpha^2\gamma\alpha^2\gamma & =
\alpha\gamma\alpha\gamma\alpha^2(\rho - \alpha\gamma\alpha)\notag\\
& =(\alpha\gamma)^2(\rho - \alpha\gamma\alpha)\rho -
\alpha\gamma\alpha\gamma\alpha^2\gamma\alpha\gamma\alpha 
\notag\\ 
& =\alpha\gamma)^2\zeta - (\alpha\gamma)^3\alpha\rho -
\alpha\gamma\alpha\gamma\alpha^2\gamma\alpha\gamma\alpha 
\notag\\ 
& = -(\alpha\gamma)^5 - (\alpha\gamma)^3\alpha^2\gamma\alpha - 
(\alpha\gamma)^3\alpha\gamma\alpha^2
-\alpha\gamma\alpha\gamma\alpha^2\gamma\alpha\gamma\alpha.\notag
\end{align}
Symmetrically we can replace \eqref{eq:B}\textrm{-I$^*$} by
\[\alpha\gamma\alpha^2(\gamma\alpha)^3 + \alpha^2(\gamma\alpha)^4 +
\alpha\gamma\alpha\gamma\alpha^2\gamma\alpha\gamma\alpha
\leqno{\eqref{eq:B}\textrm{-II$^*$}}.\]   
We substitute \eqref{eq:B}\textrm{-II} and
\eqref{eq:B}\textrm{-II$^*$} into $(\Diamond)$ and cancel, this
leaves us to show that 
\[(\alpha\gamma)^3\alpha^2\gamma\alpha 
+\alpha\gamma\alpha^2(\gamma\alpha)^3 + 
\alpha\gamma\alpha\gamma\alpha^2\gamma\alpha\gamma\alpha = 0.\]
We get this directly by premultiplying the identity $(\dagger)$ with
$\alpha\gamma$ and postmultiplying with $\gamma\alpha$.  This proves
the identity (B), and hence that $y$ is in $Z_\gr(E(\L))$.

\subsection*{Finite generation}
We want to show that $E_2$ is finitely generated over the subalgebra
generated by the central elements $\zeta$ and $\omega^2$. Let $\phi$
be a monomial in $\{\alpha, \gamma\}$. We define its length $l(\phi)$,
to be the total number of factors $\{\alpha, \gamma\}$.  As a first
goal towards finite generation, we want to express $\phi$ as a
polynomial in $\zeta, \omega^2$, where the coefficients are
polynomials in $\{\alpha, \gamma\}$, and so that we have only finitely
many coefficients. To this end, we have already seen that 
\begin{enumerate}
\item[(i)] $\omega$ and $\rho$ commute with $\alpha$,
\item[(ii)] $\omega\gamma + \gamma\omega = - \zeta$,
\item[(iii)] $\omega\rho + \rho\omega = -3\zeta\alpha^2$,
\item[(iv)] $\gamma\rho = \rho\gamma + (\gamma\alpha)^2 - (\alpha\gamma)^2$.
\end{enumerate}
Furthermore, we have that $(\gamma\alpha)^{3k}\gamma =
(-1)^k\zeta^k\gamma$ for $k\geq 1$, since $\zeta \gamma =
-(\gamma\alpha)^3\gamma$. And, $(\alpha^2\gamma)^{2k+1} =
\rho^{2k}\alpha^2\gamma = \zeta^k\alpha^2\gamma$ for $k\geq 1$, since
$\alpha^2\gamma = \omega - \alpha\gamma\alpha - \gamma\alpha^2$.
These observations are a main step towards our first goal.  At the next
step we get expressions, which have $\omega$ as a factor. Although this
is not central, we can move factors of $\omega$ to the left, as a 
consequence of the next lemma. 
\begin{lemma}
Assume $\phi = p\omega q$ where $p, q$ are monomials in $\alpha,
\gamma$. Then
\[\phi = -\omega pq + \zeta \sum \phi_i,\]
where $\phi_i$ are monomials in $\alpha, \gamma$ of length $l(\phi_i)
<l(p)+l(q)$. 
\end{lemma}
We leave the proof to the reader, only pointing out that induction and
the identity $\gamma\omega = -(\omega-\zeta)$ are used. The final
crucial step is the following. 
\begin{lemma}
Assume $\phi$ is a monomial in $\{\alpha, \beta\}$. Then $\phi$ is a
linear combination of elements of the form
\[\zeta^a\omega^b \phi_1,\]
where $\phi_1$ is a monomial in $\{\alpha, \gamma\}$ of the form
\[\alpha^i(\gamma\alpha)^r (\gamma \alpha^2)^s\gamma\alpha^j
\leqno{(*)}\]
with $r\leq 2$ and $s\leq 2$ and $i, j\leq 2$.
\end{lemma}
\begin{proof}
If suffices to express $\phi$ as a combination of elements
$\zeta^a\omega^b\phi_1$ with $\phi_1$ as in $(*)$, but for arbitrary $r$
and $s$. Then the coefficients can be reduced according to the first
Lemma. Suppose therefore that $\phi$ is not of the form as in
$(*)$. Then we can write
\[\phi = p(\alpha^2\gamma\alpha\gamma)q,\] 
where $p$ and $q$ are monomials in $\{\alpha, \gamma\}$. This is equal 
to 
\[p[\omega - \alpha\gamma\alpha\gamma\alpha -
  \gamma\alpha\gamma\alpha^2])q 
= p\omega q - p\alpha\gamma\alpha\gamma\alpha q -
p\gamma\alpha\gamma\alpha^2q 
\leqno{(**)}\] 
and $p\omega q = \omega pq + \zeta \sum \phi_i$ with $l(\phi_i) <
l(p)+l(q)$. For $pq$ we use induction. The second term in $(**)$ has
one factor $\alpha^2$ less, and in the third term of $(**)$,
$\alpha^2$ occurs further to the right. So if we start with the
rightmost $\alpha^2$ in $\phi$ and iterate the above substitution,
then we can move $\alpha^2$ completely to the right. Hence we get
terms with coefficients either of shorter length, or with fewer
$\alpha^2$ and where to the right only submonomials
$\ldots\gamma\alpha\gamma\ldots$ occur, but where the length does not
increase. The claim follows now by induction.
\end{proof}
The claim that $\L$ satisfies \textbf{(Fg)} now follows immediately. 
\end{proof}

%%%%%%%%%%%%%%%%%%%%%%%%%%%%%%%%%%%%%%%%%%%%%%%%%%%%%%%%%%%%%%%%%%%

\section{Quantum exterior algebras}\label{section:qea} 

This final section is devoted to characterizing when the quantum
exterior $k$-algebra
\[\L=k\langle x_1,x_2,\ldots,
x_n\rangle/(\{x_ix_j+q_{ij}x_jx_i\}_{i<j},\{x_i^2\}_{i=1}^n, q_{ij}\in
k^*)\] 
satisfies \textbf{(Fg)} for a field $k$. This completes the discussion
of the $\widetilde{\mathbb{Z}}_d$-case, in that we treat the case
$d=0$.

It is well-known that $\L$ is a Koszul algebra and that the Koszul
dual $E(\L)$ is given by $E(\L)=k\langle x_1,x_2,\ldots,
x_n\rangle/(\{x_ix_j-q_{ij}x_jx_i\}_{i<j}, q_{ij}\in k^*)$. To answer
our question we apply Theorem \ref{thm:Fgcharacterization}, so
our first task is to study the graded centre of $E(\L)$.

Any element in $E(\L)$ can be written uniquely as a linear combination of
the elements $\{ x_1^{t_1}x_2^{t_2}\cdots x_n^{t_n}\mid t_j\geq 0\}$.
Given that $E(\L)$ is generated in degree $1$ as an algebra over $k$, to
compute the graded centre of $E(\L)$ it is necessary and sufficient to
check graded commutation with the variables
$\{x_i\}_{i=1}^n$. Multiplication with the variables $\{x_i\}_{i=1}^n$
from the left and the right on an element $x_1^{t_1}x_2^{t_2}\cdots
x_n^{t_n}$ as above is given by:
\[ x_1^{t_1}x_2^{t_2}\cdots x_n^{t_n}x_i = 
\begin{cases}
  x_1^{t_1}x_2^{t_2}\cdots x_n^{t_n+1}, & i=n,\\
  (\prod_{j=n}^{i+1}q_{ij}^{-t_j}) x_1^{t_1}x_2^{t_2}\cdots x_{i-1}^{t_{i-1}}
  x_i^{t_i+1}x_{i+1}^{t_{i+1}}\cdots x_n^{t_n}, & i<n, 
\end{cases}\]
and 
\[ x_ix_1^{t_1}x_2^{t_2}\cdots x_n^{t_n} = 
\begin{cases}
  x_1^{t_1+1}x_2^{t_2}\cdots x_n^{t_n+1}, & i=1,\\
  (\prod_{j=1}^{i-1}q_{ji}^{-t_j}) x_1^{t_1}x_2^{t_2}\cdots x_{i-1}^{t_{i-1}}
  x_i^{t_i+1}x_{i+1}^{t_{i+1}}\cdots x_n^{t_n}, & i>1.
\end{cases}\]
\begin{prop}
The quantum exterior $k$-algebra 
\[\L=k\langle x_1,x_2,\ldots,
x_n\rangle/(\{x_ix_j+q_{ij}x_jx_i\}_{i<j},\{x_i^2\}_{i=1}^n, q_{ij}\in
k^*)\] 
satisfies \emph{\textbf{(Fg)}} if and only if all the elements
$\{q_{ij}\}_{i<j}$ are roots of unity. 
\end{prop}
\begin{proof}
First we characterize when $x_i^p$ is in $Z_\gr(E(\L))$, where $E(\L)$ is as
above. It follows immediately from the formulae above that
\[x_i^px_j = \begin{cases}
x_i^px_j, & i < j,\\
x_i^{p+1}, & i = j,\\
q_{ji}^{-p}x_jx_i^p, & i > j,
\end{cases}\]
and 
\[x_jx_i^p = \begin{cases}
q_{ij}^{-p}x_i^px_j, & i < j,\\
x_i^{p+1}, & i = j,\\
x_jx_i^p, & i > j.
\end{cases}\] 
From these formulae we obtain that $x_i^p$ is in $Z_\gr(E(\L))$ if and
only if 
\begin{enumerate}
\item[(i)] $1=(-1)^pq_{ij}^{-p}$ for $j>i$, 
\item[(ii)] $1=(-1)^p$ for $j=i$, 
\item[(iii)] $q_{ji}^{-p}=(-1)^p$ for $j<i$.
\end{enumerate}
For $\kar k\neq 2$, this is clearly equivalent to (A) $p$ is even and
(B) $\{q_{ij}\}_{j>i}$ are all $p$-th roots of unity. For $\kar k=2$,
then this is equivalent to just (B).

Suppose now that not all the elements $\{q_{ij}\}_{j>i}$ are roots of
unity for some $i$ in $\{1,2,\ldots,n-1\}$. Hence $x_i^p$ is not in
$Z_\gr(E(\L))$ for any $p\geq 1$. On the other hand $x_i^l$ is in
$E(\L)$ for any $l\geq 1$. But since $E(\L)$ is a domain and
$Z_\gr(E(\L))$ (and $E(\L)$) is multi-graded by $\mathbb{N}^n$, the
only way the elements $x_i^l$ can be generated by the graded centre
$Z_\gr(E(\L))$ is that it is one of the generators. Hence $E(\L)$ is
infinitely generated as a module over $Z_\gr(E(\L))$. Consequently
\textbf{(Fg)} is not satisfied for $\L=E(E(\L))$.

Suppose now that all the elements $\{q_{ij}\}_{j>i}$ are roots of
unity for all $i$ in $\{1,2,\ldots,n-1\}$. Suppose that each $q_{ij}$ is
a root of unity of degree $d_{ij}$. Let $N$ be the least common even
multiple of all the $d_{ij}$'s. From the calculations above we
infer that $x_i^N$ is in $Z_\gr(E(\L))$ for all $i$ in
$\{1,2,\ldots,n\}$. Then the set $\{x_1^{t_1}x_2^{t_2}\cdots
x_n^{t_n}\mid t_i < N, \forall i\in \{1,2,\ldots,n\}\}$ is a
generating set of $E(\L)$ as a module over the graded centre
$Z_\gr(E(\L))$. Hence $E(\L)$ is a finitely generated module over
$Z_\gr(E(\L))$, and therefore \textbf{(Fg)} is satisfied for
$\L=E(E(\L))$.
\end{proof}
This result is generalized by Bergh and Oppermann in \cite[Theorem
  5.5]{BO}.  
\begin{remark}
The Nakayama automorphism $\nu$ of $\L$ is given by
$\nu(x_i)=(-1)^{n-1}\frac{\prod_{j=i+1}^n
  q_{ij}}{\prod_{j=1}^{i-1}q_{ji}}x_i$ for $i=1,2,\ldots,n$. For
$n=3$, we can choose $q_{12}=q$, $q_{13}=q^{-1}$ and $q_{23}=q_{12}$,
where $q$ is not a root of unity. Then the Nakayama automorphism $\nu$
is the identity, while \textbf{(Fg)} is not satisfied. Hence the
property \textbf{(Fg)} is not linked to the Nakayama automorphism
being of finite order, or even the identity.
\end{remark}

\end{document}